\documentclass[10pt]{amsart}
\usepackage{color, graphicx, mathscinet}
\usepackage{comment}
\numberwithin{equation}{section}
\newtheorem{thm}{Theorem}[section]
\newtheorem{cor}[thm]{Corollary}
\newtheorem{lem}[thm]{Lemma}

\theoremstyle{definition}
\newtheorem{defi}{Definition}[section]
\newtheorem{example}{Example}[section]

\newcommand{\la}{\lambda}

\allowdisplaybreaks

\title{Difference operators for partitions and some applications}
\date{October 24, 2017}
\author[Guo-Niu Han]{Guo-Niu Han}
\author[Huan Xiong]{Huan Xiong$^*$}
\address{Universit\'e de Strasbourg, CNRS, IRMA UMR 7501, F-67000 Strasbourg, France}
\email{guoniu.han@unistra.fr, \quad xiong@math.unistra.fr}



\thanks{$^*$ Huan Xiong is the corresponding author.}

\begin{document}
\begin{abstract} 
Motivated by the Nekrasov-Okounkov formula on hook lengths, the first author conjectured that the Plancherel average of the $2k$-th power sum of hook lengths of partitions with size $n$ is always a polynomial of $n$ for any $k\in \mathbb{N}$. This conjecture was generalized and proved by Stanley (Ramanujan J., 23\,(1--3)\,:\,91--105, 2010). In this paper, inspired by the work of Stanley and Olshanski on the differential poset of Young lattice, we study the properties of two kinds of difference operators $D$ and $D^-$ defined on  functions of partitions. Even though the calculations for higher orders of $D$ are extremely complex, we prove that several well-known families of functions of partitions are annihilated by a power of the difference operator $D$. As an application, our results lead to several generalizations of classic results on partitions, including the marked hook formula, Stanley Theorem, Okada-Panova hook length formula, and Fujii-Kanno-Moriyama-Okada content formula. We insist that the Okada constants $K_r$ arise directly from the computation for a single partition $\lambda$, without the summation ranging over all partitions of size~$n$. 
\end{abstract}
\maketitle

\section{Introduction}\label{sec:introduction} 
The aim of this paper is to develop a formal method to discover new hook length identities of partitions and  generalize classical such identities which occur in Combinatorics, Number Theory, Representation Theory and
Mathematical Physics by difference operator technique, which is motivated by the work of Stanley \cite{stan1} and Olshanski \cite{ols1,ols2,ols3} on the differential poset of Young lattice. Our main results are the Theorems \ref{th:main1} and \ref{th:skewstanley}.

First we recall some basic definitions. We refer the reader to \cite{Macdonald,ec2} for the basic knowledge on partitions and symmetric functions. A \emph{partition} is a finite weakly decreasing sequence of positive integers $\lambda = (\lambda_1, \lambda_2, \ldots, \lambda_\ell)$. Here the integer  $|\lambda|=\sum_{1\leq i\leq \ell}\lambda_i$ is called the \emph{size} of the partition $\lambda$. A partition $\lambda$ is identified with its \emph{Young diagram}, which is a collection of boxes arranged in left-justified rows with $\lambda_i$ boxes in the $i$-th row. The \emph{content} of the box $\square=(i,j)$ in the $i$-th row and $j$-th column of the Young diagram of a partition is defined by $c_{\square}=j-i$ (see \cite{lascoux, ec2}). The \emph{hook length} of the box $\square$ in the Young diagram,  denoted by $h_{\square}$, is the number of boxes exactly to the right, or exactly above, or the box itself (the French convention for the Young diagrams is used in this paper) (see \cite{han, ec2}). For example, the Young diagram and hook lengths of the partition $(6,3,3,2)$ are illustrated in Figure \ref{fig:1}.
A \emph{standard Young tableau} of shape $\la$ is obtained by filling in the boxes of the Young diagram of $\la$ with numbers from $1$ to $|\la|$ such that the numbers strictly increase along every row and every column.
Suppose that $\lambda$ and $\mu$ are two partitions with $\lambda\supseteq \mu$, which means that the Young diagram of $\la$ contains the Young diagram of $\mu$. Denote by $f_\lambda$ (resp. $f_{\lambda/\mu}$) the number of standard Young tableaux of shape $\lambda$ (resp. $\lambda/\mu$).  Let $H_\lambda=\prod_{\square\in\lambda} h_{\square}$ be the product of all hook lengths of boxes in $\lambda$. Set $f_\emptyset=1$ and $H_\emptyset=1$ for the empty partition $\emptyset$. It is well known that (see \cite{tew, hanxiong2, frt, han2,  han4, hanxiong3, rsk, ec2})
\begin{equation}
\label{eq:hookformula}
f_\lambda = \frac{|\lambda| !}{H_\lambda} \qquad
\text{and}\qquad
\frac{1}{n!}\sum_{|\lambda| =n} f_\lambda^2=1.
\end{equation}
Here $\frac{f_\la^2}{|\la|!}$ is called the   {\it Plancherel measure} of the partition $\la$ and 
$$
\frac{1}{n!}\sum\limits_{\mid\lambda\mid= n} f_\lambda^2g(\lambda)
$$ 
is called the \emph{Plancherel average} of the function $g(\la)$ (see \cite{KJ,ols3}).
\begin{figure}
\centering
\begin{center}
\includegraphics[]{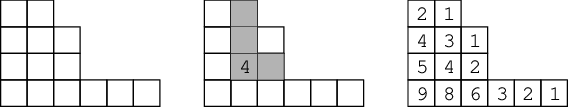}
\end{center}
\caption{The Young diagram of the partition $(6,3,3,2)$ and the hook
lengths of corresponding boxes.} \label{fig:1}
\end{figure}

Nekrasov and Okounkov \cite{no} obtained the following formula for hook lengths
$$
  \sum_{n\geq 0}\left(\sum_{|\lambda|= n}f_\lambda^2
   \prod_{\square\in\lambda} (t+h_{\square}^2)\right)\frac{x^n}{n!^2} =
   \prod_{i\geq 1}(1-x^i)^{-1-t},
$$
which was generalized and given a more elementary proof by the first author \cite{han2}. Motivated by the above formula, the first author conjectured that the Plancherel average of the power sum of hook lengths 
$$
\frac{1}{n!}\sum_{|\lambda|= n} f_\lambda^2\sum_{\square\in\lambda}h_{\square}^{2k}
$$ 
is always a polynomial of $n$ for any given positive integer $k$, which was generalized and proved by Stanley \cite{stan}.

\begin{thm}[Stanley]\label{th:hanstanley}
Let $F=F(z_1,z_2,\ldots)$ be a symmetric function of infinite variables. Then the Plancherel average 
\begin{equation}\label{eq:hanstanley}
\frac{1}{n!}\sum_{|\lambda|= n}
f_\lambda^2\, F(h_{\square}^2: {\square}\in\lambda)
\end{equation}
is a polynomial of $n$, where $F(h_{\square}^2: {\square}\in\lambda)$ means that $n$ of the variables $z_1, z_2, \ldots$ are substituted by $h_\square^2$ for $\square\in \lambda$, and all other variables by $0$.
\end{thm}
The polynomiality of \eqref{eq:hanstanley} suggested Okada to conjecture an explicit formula~\cite{stan}, which was proved by Panova \cite{panova}. Let
\begin{equation*}\label{def:S}
S(\lambda,r):=\sum_{\square\in\lambda}\prod_{1\leq j\leq
r}(h_{\square}^2-j^2)
\end{equation*}
and
\begin{equation*}
    K_r:=\frac{(2r)!(2r+1)!}{r! (r+1)!^2}.
\end{equation*}
The sequence $(K_0=1,\, K_1=3,\, K_2=40,\, K_3=1050,\, \ldots)$ appears as $A204515$ in the \emph{On-Line Encyclopedia of Integer Sequences} \cite{OEIS:A204515}.
\begin{thm}[Okada-Panova \cite{panova}]\label{th:okadapanova}
For each positive integer $n$ we have
\begin{equation}\label{eq:okadapanova}
n!  \sum_{|\lambda|=n}\frac{S(\lambda, r)}{H_\lambda^2} = K_r
\binom{n}{r+1}.
\end{equation}
\end{thm}

In this paper, we study two kinds of difference operators $D$
and $D^-$ defined on functions of partitions, motivated by the work of Stanley \cite{stan1} and Olshanski \cite{ols1,ols2,ols3} on differential poset of Young lattice. As applications, we will generalize Stanley Theorem, Okada-Panova hook length formula and obtain other more general results by studying the difference operator $D$ on each single summand $F(h_{\square}^2: {\square}\in\lambda)$. As will be seen in Corollary \ref{th:DS} the constants $K_r$ arise directly from the computation for a single partition $\lambda$, without the summation ranging over all partitions of size~$n$. 

The differential poset of Young lattice was first introduced in $1988$ by Stanley \cite{stan1}. In his paper, Stanley studied the following two operators for partitions:
\begin{equation*}
    T_1(\lambda):=\sum_{\lambda^{+}}\lambda^{+}
\text{\quad and\quad}
T_2(\lambda):=\sum_{\lambda^-}\lambda^-,
\end{equation*}
where $\lambda^{+}$ (resp. $\lambda^-$) ranges over all partitions
obtained by adding (resp. removing) a box to (resp. from) $\lambda$.
Many remarkable results on Young lattice and partitions were obtained by this technique \cite{ols1,ols2,ols3,petrov,stan1}. For example, Olshanski \cite{ols3} gave a proof of the content case of Stanley Theorem by replacing $\la$ by certain functions  $\frac{g(\la)}{H_\la}$ related to contents of partitions in the definition of $T_1$ and $T_2$. In this paper, we will give a systematic application of Stanley and Olshanski's ideas, which demonstrates the power of the difference operator approach.

\begin{defi} 
Let $g(\lambda)$ be a function defined on partitions. The {\it difference operators} $D$ and $D^-$  are defined by
\begin{align*}
    Dg(\lambda)&=\sum_{\lambda^{+}}g(\lambda^{+})-g(\lambda)\\
\noalign{{\noindent and}}
D^-g(\lambda)&=| \lambda |\, g(\lambda)- \sum_{\lambda^-}g(\lambda^-),
\end{align*}
where $\lambda^{+}$ (resp. $\lambda^-$) ranges over all partitions obtained by adding (resp. removing) a box to (resp. from) $\lambda$. Higher-order difference operators for $D$ are defined by induction $D^0g:=g$ and $D^k g:=D(D^{k-1} g)$ ($k\geq 1$). Also, we write $D g(\mu) := D g(\lambda) |_{\lambda=\mu}$ for a fixed partition $\mu$.
\end{defi}
We will show in Lemma \ref{th:DD-} that the two difference operators $D$ and $D^-$ satisfy the simple noncommutative law $DD^--D^-D=D.$

The functions of partitions which are annihilated by a power of the difference operator $D$ is crucial in our study.

\begin{defi}
A function $g(\lambda)$ of partitions is called a {\it $D$-polynomial on partitions}, if there exists a nonnegative integer $r$ such that $D^{r+1} \bigl(g(\lambda)/H_\lambda\bigr)=0$ for every partition $\lambda$. The minimal $r$ satisfying this condition is called the {\it degree} of~$g(\lambda)$.
\end{defi}

In this paper, we will show that several types of functions of partitions, such as the power sums of hook lengths and contents, are $D$-polynomials.

Our two main theorems are stated next.

\begin{thm}\label{th:main1}
For each power sum symmetric function $p_\nu(z_1, z_2, \ldots)$ of infinite variables indexed by the partition $\nu=(\nu_1, \nu_2, \ldots, \nu_\ell)$,
the function $p_\nu(h_{\square}^2: {\square}\in\lambda)$ of partition $\lambda$ is a $D$-polynomial with degree at most $|\nu|+\ell$.
\end{thm}

\medskip

\begin{thm} \label{th:skewstanley}
Let $\mu$ be a given partition and $k$ be a nonnegative integer. For each power sum symmetric function $p_\nu(z_1, z_2, \ldots)$
we have
$$
\sum_{|\lambda/\mu|=n}f_{\lambda/\mu}
D^k\Bigl(\frac{p_\nu(h_{\square}^2:\square\in
\lambda)}{H_\lambda}\Bigr)=\sum_{0\leq i\leq |\nu|+\ell-k}d_{i+k}\binom{n}{i}
$$ 
is a polynomial of $n$, where
$$
d_i=D^i(\frac{p_\nu(h_{\square}^2:\square\in \mu)}{H(\mu)}).
$$ 
In particular, let $k=0$. Then
\begin{equation}
    \frac{1}{(n+|\mu|)!}\sum_{|\lambda/\mu|= n}f_{\lambda}
f_{\lambda/\mu}p_\nu(h_{\square}^2: {\square}\in\lambda)
\end{equation}
is a polynomial of $n$ with degree at most $|\nu|+\ell$. Furthermore,
$$
\frac{1}{(n+|\mu|)!}\sum_{|\lambda/\mu|= n}f_{\lambda}
f_{\lambda/\mu}F(h_{\square}^2: {\square}\in\lambda)
$$
is a polynomial of $n$ for any given partition $\mu$ and any given symmetric function $F$.
\end{thm}

Theorem \ref{th:main1} is difficult to prove, since
the calculations for higher orders of $D$ are extremely complex.  
We have to make a full study of  
a large family of $D$-polynomials. 
In Example \ref{ex:big}, we see that $D^3g(1)$ is equal to
a sum of some fractions. Theorem \ref{th:main1} claims that
the later sum can be annihilated. 

Let us give some applications first. Knowing the polynomiality for some certain functions gets us closer to explicit formulas. By Theorem \ref{th:skewstanley}  with $\mu=\emptyset$, we derive Han-Stanley Theorem. In Section \ref{sec:okadapanova} we prove the following corollary, and show that Okada-Panova hook length formula can be derived by Corollary \ref{th:DS}.

\begin{cor}\label{th:DS}
For each nonnegative integer $r$, the function $S(\lambda, r)$ of partitions is a $D$-polynomial with degree $r+1$. More precisely,
\begin{align}
    H_\lambda D^{r}\Bigl(\frac{S(\lambda,r)}{H_\lambda}\Bigr)&= K_r |\lambda|,\label{eq:DS0} \\
    H_\lambda D^{r+1}\Bigl(\frac{S(\lambda,r)}{H_\lambda}\Bigr)&= K_r, \label{eq:DS1}\\
    H_\lambda D^{r+2}\Bigl(\frac{S(\lambda,r)}{H_\lambda}\Bigr)&= 0. \label{eq:DS2}
\end{align}
\end{cor}

The special case $r=1$ of Okada-Panova hook length formula is usually called {\it marked hook formula} \cite{han}:
\begin{equation}\label{eq:marked}
    \sum_{|\lambda|=n}\frac{f_\lambda}{H_\lambda} \, S(\lambda, 1)
=3\binom{n}{2}.
\end{equation}
In Section \ref{sec:carde} we obtain the following generalization of \eqref{eq:marked}. Notice that we couldn't find such nice explicit formulas for general $S(\lambda,r)$ since $D^{i}\Bigl(\frac{S(\lambda,r)}{H_\lambda}\Bigr)$ doesn't have nice expression for general $i\leq r-1$. 
\begin{thm}[Skew marked hook formula]\label{th:skewmarked}
Let $\mu$ be a given partition. For every $n\geq |\mu|$ we have
\begin{equation}\label{eq:skewmarked}
    \sum_{|\lambda|=n,\, \lambda \supset \mu} \frac{H_\mu
    f_{\lambda/\mu}}{H_\lambda}\, \Bigl(S(\lambda,1)-S(\mu,1)\Bigr) = \frac 32
\,  (n-|\mu|)\, (n+|\mu|-1).
\end{equation}
\end{thm}

\medskip
Recall that the \emph{content} of the box $\square=(i,j)$ in the Young diagram of a partition is defined by $c_{\square}=j-i$ (see \cite{lascoux, ec2}). Let
$$
C(\lambda,r):=\sum_{\square\in\lambda}\prod_{0\leq j\leq
r-1}(c_{\square}^2-j^2).
$$
The following similar results are obtained for contents in Section \ref{sec:content}.
\begin{thm}\label{th:DC}
For each positive integer $r$, the function $C(\lambda, r)$ of partitions is a $D$-polynomial of degree $r+1$. More precisely,
\begin{align}
    H_\lambda D^{r}\Bigl(\frac{C(\lambda,r)}{H_\lambda}\Bigr)&= \frac{(2r)!}{(r+1)!} |\lambda|,\label{eq:DC0} \\
    H_\lambda D^{r+1}\Bigl(\frac{C(\lambda,r)}{H_\lambda}\Bigr)&= \frac{(2r)!}{(r+1)!}, \label{eq:DC1}\\
    H_\lambda D^{r+2}\Bigl(\frac{C(\lambda,r)}{H_\lambda}\Bigr)&= 0. \label{eq:DC2}
\end{align}
\end{thm}

\begin{thm}[Fujii-Kanno-Moriyama-Okada \cite{fkmo}]\label{th:content}
For each positive integer $n$ we have
\begin{equation*}
    n!  \sum_{|\lambda|=n}\frac{C(\lambda, r)}{H_{\lambda}^2} =
    \frac{(2r)!}{(r+1)!} \binom{n}{r+1}.
\end{equation*}
\end{thm}

\begin{thm}[Skew marked content formula]\label{th:skewmarkedcontent}
Let $\mu$ be a given partition. For every $n\geq |\mu|$ we have
\begin{equation}\label{eq:skewmarkedcontent}
    \sum_{|\lambda|=n,\, \lambda \supseteq \mu} \frac{H_\mu
    f_{\lambda/\mu}}{H_\lambda}\, \Bigl(C(\lambda,1)-C(\mu,1)\Bigr) = \frac 12
\,  (n-|\mu|)\, (n+|\mu|-1).
\end{equation}
\end{thm}

The rest of the paper is arranged in the following way. In Section \ref{sec: diffope} we study the general properties for the difference operators $D$ and $D^-$. The connection between  difference operator $D$ and the Plancherel average of functions of partitions will be established in Section \ref{sec:main2}.  In Sections \ref{sec:shifted} and \ref{sec:carde}, we study two specific families of $D$-polynomials arising from the work of the first author on the shifted parts of partitions \cite{han3} and the work of Carde,  Loubert,  Potechin and  Sanborn \cite{clps}. In section \ref{sec:Dpoly}, we study the properties of the functions $q_\nu(\lambda)$ needed in the proof of our main results. Later, we prove the main results Theorems \ref{th:main1} and \ref{th:skewstanley} in Section~\ref{sec:main1}. Finally, we prove and generalize the Okada-Panova hook length formula and the Fujii-Kanno-Moriyama-Okada content formula by difference operator technique in Sections \ref{sec:okadapanova} and \ref{sec:content} respectively.

\section{Difference operators for partitions} \label{sec: diffope} 
The difference operators $D$ and $D^-$ defined in Section \ref{sec:introduction} are our fundamental tools for studying hook length formulas. This section is devoted to establish some basic properties.
It is obvious that $D$ and $D^-$ are  linear operators.

\begin{lem}
Let  $\lambda$ be a partition and $g_1, g_2$ be two functions of partitions. The following identities hold for all $a_1,a_2\in \mathbb{R}$ $:$
\begin{align*}
    D(a_1g_1+a_2g_2)(\lambda) &= a_1Dg_1(\lambda)+a_2Dg_2(\lambda),\\
D^-(a_1g_1+a_2g_2)(\lambda) &= a_1D^-g_1(\lambda)+a_2D^-g_2(\lambda).
\end{align*}
\end{lem}

The function $H_\la$ is a $D$-polynomial with degree $0$.
\begin{lem} \label{th:D+}
For each partition $\lambda$ we have
$$
D\Bigl(\frac{1}{H_\la}\Bigr)=0.
$$
\end{lem}
\begin{proof}
Let $n=| \la |$. Consider the following two sets related to standard Young tableaux (written as ``SYT'' for simplicity)
\begin{align*}
    A&=\{(i, T) : 1\leq i\leq n+1, T\ \text{is an SYT of shape } \lambda
\},\\
B&=\{(\lambda^+, T^+) : |\lambda^+/\lambda|=1,
T^+\ \text{is an SYT of shape } \lambda^+\}.
\end{align*}
Let $(i, T)\in A$. First we increase every entry which is greater than or equal to $i$ by one in $T$. Then we use the Robinson-Schensted-Knuth algorithm \cite{rsk} to insert the integer $i$ into $T$ to get a new SYT $T^+$. Let $\lambda^+$ be the shape of $T^+$. We have $|\lambda^+/\lambda|=1$, so that $(\lambda^+, T^+)\in B$. It is easy to see that this is a bijection between sets $A$ and $B$. The cardinalities of $A$ and $B$ are $(n+1)f_\lambda$ and $\sum_{\lambda^+} f_{\lambda^+}$ respectively. Hence we obtain
$$
(n+1)f_{\lambda}=\sum_{\lambda^{+}}f_{\lambda^{+}}.
$$ 
This implies 
\begin{equation*}
    D\Bigl(\frac{1}{H_\lambda}\Bigr) =
\sum_{\lambda^{+}}\frac{1}{H_{\lambda^{+}}}-\frac{1}{H_\lambda}
 =
\frac{1}{(n+1)!}\Bigl(\sum_{\lambda^{+}}f_{\lambda^{+}}-(n+1)f_{\lambda}\Bigr)
= 0.\qedhere
\end{equation*}
\end{proof}

For the difference operator $D^-$ we obtain the following similar results.
\begin{lem}\label{th:D-}
Let $g(\lambda)$ be a function of partitions. Then $D^-g(\lambda)=0$ for every  partition $\lambda$ if and only if
$$
g(\lambda)=\frac{a}{H_\lambda}
$$ 
for some constant $a$.
\end{lem}
\begin{proof}
By the definition of SYTs it is obvious that $f_{\lambda}=\sum_{\lambda^-}f_{\lambda^-}$. Thus,
\begin{equation}\label{eq:D-}
    D^-\Bigl(\frac{a}{H_\lambda}\Bigr)
= \frac{a|\lambda|}{H_\lambda}- \sum_{\lambda^-}\frac{a}{H_{\lambda^-}}
=
\frac{a}{(|\lambda|-1)!}\Bigl(f_{\lambda}-\sum_{\lambda^-}f_{\lambda^-}\Bigr)
= 0.
\end{equation}
On the other hand, $D^-g(\lambda)=0$ implies $| \lambda |\, g(\lambda)=\sum_{\lambda^-}g(\lambda^-)$. Let $a=g(\emptyset)$ where~$\emptyset$ is the empty partition. By induction and \eqref{eq:D-} we obtain $g(\lambda)=\frac{a}{H_\lambda}.$
\end{proof}

Notice that it is not easy to determine the functions $g(\lambda)$ under the condition $Dg(\lambda)=0$ for every partition $\lambda$. For example, by \eqref{eq:DS0} and the following Lemma \ref{th:Dbino} we obtain
$$
D\Biggl(\frac{\sum\limits_{{\square}\in\lambda}(h_{\square}^2-1)-3\binom{|\la|}{2}
}{H_\lambda}\Biggr)=0.
$$

\begin{lem} \label{th:Dbino}
  For each positive integer $r$ we have
 $$D\Bigl(\frac{\binom{|\la|}{r}}{H_\lambda}\Bigr)=\frac{\binom{|\la|}{r-1}}{H_\lambda}
\text{\qquad and\qquad}
D^-\Bigl(\frac{\binom{|\la|}{r}}{H_\lambda}\Bigr)=\frac{r\binom{|\la|}{r}}{H_\lambda}.$$
\end{lem}
\begin{proof}
Let $n=|\la|$. By Lemmas \ref{th:D+} and \ref{th:D-} we obtain
\begin{align*}
    D\Bigl(\frac{\binom{n}{r}}{H_\lambda}\Bigr) &=
\sum_{\lambda^{+}}\frac{\binom{n+1}{r}}{H_{\lambda^{+}}}-
\frac{\binom{n}{r}}{H_\lambda}
 =
\frac{\binom{n+1}{r}-\binom{n}{r}}{H_\lambda}
= \frac{\binom{n}{r-1}}{H_\lambda},\\
    D^-\Bigl(\frac{\binom{n}{r}}{H_\lambda}\Bigr)
    &= \frac{n\binom{n}{r}}{H_\lambda}- \sum_{\lambda^-}\frac{\binom{n-1}{r}}{H_{\lambda^-}}
 =
\frac{n\binom{n}{r}-n\binom{n-1}{r}}{H_\lambda}
= \frac{r\binom{n}{r}}{H_\lambda}.\qedhere
\end{align*}
\end{proof}

In fact, we obtain the following more general results for $D$ and $D^-$.
\begin{lem}\label{th:leibniz1}
For each function $g$ defined on partitions we obtain
\begin{align*}
    D\Bigl(\frac{g(\lambda)}{H_\lambda}\Bigr)&=
\sum_{\lambda^{+}}\frac{g(\lambda^{+})-g(\lambda)}{H_{\lambda^+}},\\
\noalign{\noindent and}
D^-\Bigl(\frac{g(\lambda)}{H_\lambda}\Bigr)&=
\sum_{\lambda^{-}}\frac{g(\lambda)-g(\lambda^-)}{H_{\lambda^-}}.
\end{align*}
\end{lem}
\begin{proof}
By Lemmas \ref{th:D+} and \ref{th:D-} we have
\begin{align*}
    D\Bigl(\frac{g(\lambda)}{H_\lambda}\Bigr) &=
\sum_{\lambda^{+}}\frac{g(\lambda^{+})}{H_{\lambda^{+}}}
-\frac{g(\lambda)}{H_\lambda}
=
\sum_{\lambda^{+}}\frac{g(\lambda^{+})-g(\lambda)}{H_{\lambda^+}},\\
D^-\Bigl(\frac{g(\lambda)}{H_\lambda}\Bigr)&=
|\lambda|\frac{g(\lambda)}{H_\lambda}-
\sum_{\lambda^-}\frac{g(\lambda^-)}{H_{\lambda^-}}
=
\sum_{\lambda^{-}}\frac{g(\lambda)-g(\lambda^-)}{H_{\lambda^-}}.\qedhere
\end{align*}
\end{proof}

\begin{lem}[Leibniz's rule] \label{th:leibniz}
Let $g_1,g_2,\cdots,g_r$ be functions defined on partitions.  We
have
\begin{align*}
D\Bigl(\frac{\prod_{1\leq j\leq r}{g_j(\lambda)}}{H_\lambda}\Bigr)
&=\sum_{\lambda^{+}}\sum\limits_{(*)}
\frac 1{H_{\lambda^+}}
\Bigl({\prod_{k\in
    A}\bigl(g_k(\lambda^+)-g_k(\lambda)\bigr) \prod_{l\in B}
g_l(\lambda)}\Bigr) \\
\noalign{\noindent and}
D^-\Bigl(\frac{\prod_{1\leq j\leq r}{g_j(\lambda)}}{H_\lambda}\Bigr)
&=-\sum_{\lambda^{-}}
\sum\limits_{(*)}
\frac 1{H_{\lambda^-}}\Bigl(
{\prod_{k\in A} \bigl(g_k(\lambda^-)-g_k(\lambda)\bigr)\prod_{l\in B}
g_l(\lambda)}
\Bigr),
\end{align*}
where $[r]:=\{1,2,\cdots,r\}$ and the sum $(*)$ ranges over all
pairs $(A,B)\subset [r]\times [r]$ such that $A\cup B=[r],\,  A\cap
B=\emptyset$ and $A\neq \emptyset$. 
\goodbreak
In particular,
\begin{align*}D\Bigl(\frac{g_1(\lambda)g_2(\lambda)}{H_\lambda}\Bigr)&=
g_1(\lambda)D\Bigl(\frac{g_2(\lambda)}{H_\lambda}\Bigr)+
g_2(\lambda)D\Bigl(\frac{g_1(\lambda)}{H_\lambda}\Bigr)\\
& \quad +
\sum_{\lambda^{+}}
\frac 1{H_{\lambda^+}}
\bigl(g_1(\lambda^{+})-g_1(\lambda)\bigr)
\, \bigl(g_2(\lambda^{+})-g_2(\lambda)\bigr)
\end{align*}
\goodbreak
and
\begin{align*}
D^-\Bigl(
\frac{g_1(\lambda)g_2(\lambda)}{H_\lambda}\Bigr)&=
g_1(\lambda)D^-\Bigl(\frac{g_2(\lambda)}{H_\lambda}\Bigr)+
g_2(\lambda)D^-\Bigl(\frac{g_1(\lambda)}{H_\lambda}\Bigr)
\\
&\quad -
\sum_{\lambda^{-}}
\frac 1{H_{\lambda^-}}
\bigl(g_1(\lambda)-g_1(\lambda^-)\bigr)
\, \bigl(g_2(\lambda)-g_2(\lambda^-)\bigr)
.
\end{align*}
\end{lem}
\begin{proof}
    By Lemma \ref{th:leibniz1} we have
\begin{align*}
    D\Bigl(\frac{\prod_{1\leq j\leq
r}{g_j(\lambda)}}{H_\lambda}\Bigr)
&=
\sum_{\lambda^{+}}
\frac 1{H_{\lambda^{+}}}
\Bigl({\prod_{1\leq j\leq
r}{g_j(\lambda^{+})}-\prod_{1\leq j\leq
r}{g_j(\lambda)}}\Bigr)
\\
&=
\sum_{\lambda^{+}}
\frac 1{H_{\lambda^+}}
\Bigl({\prod_{1\leq j\leq
    r}\bigl(g_j(\lambda)+(g_j(\lambda^+)-g_j(\lambda))\bigr)-
    \prod_{1\leq j\leq
r}{g_j(\lambda)}}\Bigr)
\\
&= \sum_{\lambda^{+}}
\sum_{(*)}
\frac 1{H_{\lambda^+}}
\Bigl({\prod_{k\in
A} \bigl(g_k(\lambda^+)-g_k(\lambda)\bigr)\prod_{l\in B}
g_l(\lambda)}\Bigr).
\end{align*}
The proof for $D^-$ is similar.
\end{proof}

For higher-order difference operators, we have the following result.

\begin{lem} \label{th:Dkbino}
Suppose that $k$ is a nonnegative integer. Let $n=|\lambda|$. Then
we have
\begin{equation}\label{eq:Dkbino}
D^k\Bigl(\binom{n}{r}g(\lambda)\Bigr)=
\sum_{i=0}^k\binom{k}{i}\binom{n+i}{r-k+i}
D^ig(\lambda).
\end{equation}
\end{lem}
\begin{proof}
First we have
\begin{align}\label{eq:Dbinog}
D\Bigl(\binom{n+j}{r}g(\lambda)\Bigr) &=
\sum_{\lambda^{+}}\binom{n+1+j}{r}g(\lambda^{+})
-\binom{n+j}{r}g(\lambda)
\nonumber\\
&=\binom{n+1+j}{r}Dg(\lambda)
+\binom{n+j}{r-1}g(\lambda).
\end{align}
We prove \eqref{eq:Dkbino} by induction. The case $k=0,1$ is trivial by \eqref{eq:Dbinog}. Assume that the lemma is true for some $k\geq 1$, then
\begin{align*}
    &\quad  D\Bigl(D^k\bigl(\binom{n}{r}g(\lambda)\bigr)\Bigr)\\
    &=
\sum_{i=0}^k\binom{k}{i}D\Bigl(\binom{n+i}{r-k+i}D^ig(\lambda)\Bigr)
\\
&= \sum_{i=0}^k\binom{k}{i}
\Bigl(\binom{n+1+i}{r-k+i}D^{i+1}g(\lambda)
+\binom{n+i}{r-k+i-1}D^ig(\lambda)\Bigr)
\\
&=\sum_{i=1}^{k+1}\binom{k}{i-1}
\binom{n+i}{r-k+i-1}D^{i}g(\lambda)
+\sum_{i=0}^{k}\binom{k}{i}
\binom{n+i}{r-k+i-1}D^{i}g(\lambda)
\\
&= \sum_{i=0}^{k+1}\binom{k+1}{i}
\binom{n+i}{r-k+i-1}D^{i}g(\lambda).\qedhere
\end{align*}
\end{proof}
\goodbreak

\begin{lem} \label{th:DD-}
The two difference operators $D$ and $D^-$ are noncommutative, and satisfy
$$DD^--D^-D=D.$$
\end{lem}
\begin{proof}
If $(\lambda^{+})^-\neq \lambda,$ then $(\lambda^{+})^-=\lambda\cup\{\square_1\}\setminus \{\square_2\}$ for some boxes $\square_1 \neq \square_2$. This means that we can switch the order of adding $\square_1$ and removing $\square_2$ and get the same partition $(\lambda\setminus\{\square_2\})\cup\{\square_1\} \in \{(\lambda^-)^+:(\lambda^-)^+\neq \lambda\}.$ Consequently,
\begin{equation}
\{(\lambda^{+})^-: (\lambda^{+})^-\neq \lambda\}=\{(\lambda^-)^+:(\lambda^-)^+\neq
\lambda\}.
\end{equation}
For a given partition, the number of ways to add a box minus the number of ways to remove a box always equals $1$. Thus
\begin{align*}
\sum_{(\lambda^{+})^-}g\bigl((\lambda^{+})^-\bigr) -
\sum_{(\lambda^-)^+}g\bigl((\lambda^-)^+\bigr)=g(\lambda).
\end{align*}
By definition of $D$ and $D^-$, we have
\begin{align*}
DD^-g(\lambda)&= \sum_{\lambda^{+}}D^-g(\lambda^{+})-D^-g(\lambda)
\\
&=\sum_{\lambda^{+}}|\lambda^+|g(\lambda^{+})
-\sum_{(\lambda^{+})^-}g\bigl((\lambda^{+})^-\bigr) -|\lambda|
g(\lambda) + \sum_{\lambda^-}g(\lambda^-)\\
\noalign{\noindent and}
D^-Dg(\lambda)&= | \lambda | Dg(\lambda)-
\sum_{\lambda^-}Dg(\lambda^-)
\\
&=|\lambda|\sum_{\lambda^{+}}g(\lambda^{+})-|\lambda|
g(\lambda) -\sum_{(\lambda^-)^+}g\bigl((\lambda^-)^+\bigr)  +
\sum_{\lambda^-}g(\lambda^-).
\end{align*}
The above three identities yield
$$
DD^-g(\lambda)-D^-Dg(\lambda)
=\sum_{\lambda^{+}}g(\lambda^{+})-g(\lambda) =Dg(\lambda).\qedhere
$$
\end{proof}
\section{Telescoping sum for partitions} \label{sec:main2} 
In this section, we build the connection between the difference operator $D$ and the Plancherel average of functions of partitions. The main result in this section is Theorem \ref{th:main2}.
\begin{lem} \label{th:telescope}
For each given partition $\mu$ and function  $g$ of partitions, let
\begin{align*}
A(n)&:=\sum_{|\lambda/\mu|=n}f_{\lambda/\mu}g(\lambda)\\
\noalign{\noindent and}
B(n)&:=\sum_{|\lambda/\mu|=n}f_{\lambda/\mu}Dg(\lambda).
\end{align*}
Then
\begin{equation}\label{eq:AAB}
A(n)=A(0)+\sum_{k=0}^{n-1}B(k).
\end{equation}
\end{lem}
\begin{proof}
By the definition of the operator $D$,
$$\sum_{\lambda^{+}}g(\lambda^{+})=g(\lambda)+Dg(\lambda).$$
Summing the above equality  over all SYTs $T$ of shape $\lambda/\mu$
with $|\lambda/\mu|=n$, we have
\begin{align*}
A(n+1)&=A(n)+B(n).
\end{align*}
By iteration we obtain \eqref{eq:AAB}.
\end{proof}

\begin{example}
Let $g(\lambda)=1/H_\lambda$. Then $Dg(\lambda)=0$ by Lemma \ref{th:D+}. The two quantities defined in Lemma \ref{th:telescope} are:
$$
A(n)=\sum_{|\lambda/\mu|=n}  \frac{{f_{\lambda/ \mu}}}{H_\lambda}
\text{\qquad and\qquad}
B(n)=0.
$$
Consequently,
\begin{equation}\label{eq:skewhook}
\sum_{|\lambda/\mu|=n}  \frac{{f_{\lambda/ \mu}}}{H_\lambda} = \frac 1{H_\mu}.
\end{equation}
\end{example}
In particular, we derive the second identity in \eqref{eq:hookformula}
by letting $\mu=\emptyset$.

\medskip

\begin{thm}\label{th:main2}
Let $g$ be a function of partitions and $\mu$ be a given partition. Then we have
\begin{equation}\label{eq:main2}
\sum_{|\lambda/\mu|=n}f_{\lambda/\mu}g(\lambda)=\sum_{k=0}^n\binom{n}{k}D^kg(\mu)
\end{equation}
and
\begin{equation}\label{eq:main2b}
D^ng(\mu)=\sum_{k=0}^n(-1)^{n+k}\binom{n}{k}\sum_{|\lambda/\mu|=k}f_{\lambda/\mu}g(\lambda).
\end{equation}
In particular, if there exists some positive integer $r$ such that $D^{r+1} g(\lambda)=0$ for every partition $\lambda$, then the left-hand side of \eqref{eq:main2} is a polynomial of $n$ with degree at most $r$.
\end{thm}
\begin{proof}
First, we prove \eqref{eq:main2} by induction. The case $n=0$ is trivial. Assume that \eqref{eq:main2} is true for some nonnegative integer $n$. Then by the proof of Lemma \ref{th:telescope} we have
\begin{align*}
\sum_{|\lambda/\mu|=n+1}f_{\lambda/\mu}g(\lambda)&=
\sum_{|\nu/\mu|=n}f_{\nu/\mu}g(\nu)+
\sum_{|\nu/\mu|=n}f_{\nu/\mu}Dg(\nu)
\\&=\sum_{k=0}^n\binom{n}{k}D^kg(\mu)+\sum_{k=0}^n\binom{n}{k}D^{k+1}g(\mu)
\\&=\sum_{k=0}^{n+1}\binom{n+1}{k}D^kg(\mu).
\end{align*}

Finally, Identity \eqref{eq:main2b} is proved by the M\"obius inversion formula \cite{Rota1964}.
\end{proof}

\begin{example}\label{ex:big}
Let $g(\lambda)=\frac{\sum_{{\square}\in\lambda}h_{\square}^2}{H_\lambda}$, $\mu=(1)$, $n=0,1,2,3$ in Identity \eqref{eq:main2b}.  Note that 
\begin{align*}
	f_{(1)/(1)}&=1; \\
\ f_{(2)/(1)}&=1,  &\ f_{(11)/(1)}&=1,\\ 
						f_{(3)/(1)}&=1, &	\ f_{(111)/(1)}&=1,\ &f_{(21)/(1)}&=2, \\
	f_{(4)/(1)}&=1, &\ f_{(1111)/(1)}&=1,\ & f_{(31)/(1)}&=3,\ &f_{(211)/(1)}&=3,& f_{(22)/(1)}=2.
\end{align*}
Then we have 
\begin{align*}
	D^0 g(1) &= (-1)^{0+0}\binom{0}{0}f_{(1)/(1)}g(1)=g(1)=1;\\
D^1 g(1)&= (-1)^{1+1}\binom{1}{1}\left(f_{(2)/(1)}g(2)+f_{(11)/(1)}g(11)\right)+(-1)^{1+0}\binom{1}{0}f_{(1)/(1)}g(1)
\\&=g(2)+g(11)-g(1)
\\&=\frac{5}{2}+\frac{5}{2}-1=4;\\
D^2 g(1)&= (-1)^{2+2}\binom{2}{2}\left(f_{(3)/(1)}g(3)+f_{(111)/(1)}g(111)+f_{(21)/(1)}g(21)\right)
\\&\qquad+(-1)^{2+1}\binom{2}{1}\left(f_{(2)/(1)}g(2)+f_{(11)/(1)}g(11)\right)
\\&\qquad+(-1)^{2+0}\binom{1}{0}f_{(1)/(1)}g(1)
\\&
= g(3)+g(111)+2g(21)-2g(2)-2g(11)+g(1)
\\&
=
\frac{7}{3}+\frac{7}{3}+2\cdot\frac{11}{3}-2\cdot\frac{5}{2}-2\cdot\frac{5}{2}+1
=3;\\
D^3 g(1)&=(-1)^{3+3}\binom{3}{3}\bigl(f_{(4)/(1)}g(4)+f_{(1111)/(1)}g(1111)+f_{(31)/(1)}g(31)
\\&\qquad+f_{(211)/(1)}g(211)+f_{(22)/(1)}g(22)\bigr)
\\&\qquad +(-1)^{3+2}\binom{3}{2}\left(f_{(3)/(1)}g(3)+f_{(111)/(1)}g(111)+f_{(21)/(1)}g(21)\right)
\\&\qquad+(-1)^{3+1}\binom{3}{1}\left(f_{(2)/(1)}g(2)+f_{(11)/(1)}g(11)\right)
\\&\qquad+(-1)^{3+0}\binom{3}{0}f_{(1)/(1)}g(1)
\\&
=g(4)+g(1111)+3g(31)+3g(211)+2g(22)
\\&\qquad -3g(3)-3g(111)-6g(21)+3g(2)+3g(11)-g(1)
\\&
=
\frac{5}{4}+\frac{5}{4}+3\cdot\frac{11}{4}+3\cdot\frac{11}{4}+2\cdot\frac{3}{2}
\\&\qquad
-3\cdot\frac{7}{3}-3\cdot\frac{7}{3}-6\cdot\frac{11}{3}+3\cdot\frac{5}{2}+3\cdot\frac{5}{2}-1
=0.
\end{align*}
\end{example}

\section{Shifted parts of partitions} \label{sec:shifted} 
In this section, we will show that some certain functions related to shifted parts of partitions are $D$-polynomials, which is motivated by the work of the first author on hook lengths and symmetric functions \cite{han3}.

Suppose that $\lambda = (\lambda_1, \lambda_2, \ldots,
\lambda_\ell)$ is a partition with size $n$.  Let
$$
\varphi_\lambda(z)= \prod_{i=1}^n(z+n+\lambda_i-i),
$$
where $\lambda_i=0$ for $i\geq \ell+1$. The following theorem is the main result in this section.

\begin{thm}\label{th:shifted}
Suppose that $z$ is a formal parameter. For each partition $\lambda$ we have
$$
D\Bigl(\frac{\varphi_\lambda(z)}{H_\lambda}\Bigr)=\frac{z\varphi_\lambda(z+1)}{H_\lambda}.
$$
\end{thm}

Theorem \ref{th:shifted} has several direct corollaries.

\begin{cor} \label{th:shifted1} 
Suppose that $z$ is a formal parameter and $r$ is a nonnegative integer. For each partition $\lambda$ we have
$$
D^{r+1}\Bigl(\frac{\varphi_\lambda(z)}{H_\lambda}\Bigr)=\frac{z(z+1)\cdots(z+r)\varphi_\lambda(z+r+1)}{H_\lambda}.
$$
In particular,  $\varphi_\lambda(-r)$ is a $D$-polynomial with degree at most $r$, or equivalently,
$$
D^{r+1}\Bigl(\frac{\varphi_\lambda(-r)}{H_\lambda}\Bigr)=0.
$$
\end{cor}

By Corollary \ref{th:shifted1} and Theorem \ref{th:main2} we obtain

\begin{cor}
\label{th:shifted2} Suppose that $r$ is a nonnegative integer and $\mu$ is a given partition. Then we have
\begin{equation}\label{eq:shifted}
\sum_{|\lambda/\mu|=n}f_{\lambda/\mu}\frac{\varphi_\lambda(-r)}{H_\lambda}=
\sum_{k=0}^{r}\binom{n}{k}D^k\Bigl(\frac{\varphi_\mu(-r)}{H_\mu}\Bigr)
\end{equation} 
is a polynomial of $n$ with degree at most $r$.
\end{cor}

To prove Theorem \ref{th:shifted}, we need the following lemma proved by the first author in \cite {han3}.

\begin{lem}[(2.2) of \cite {han3}] \label{th:shifted3}
 Suppose that $\lambda$ is a partition and $\lambda_i>\lambda_{i+1}$ for some integer $i$.
Then
$$
 \frac{H_\lambda}{
H_{\lambda^*}}=\frac{\prod_{j=1}^n(i-\lambda_i+1+\lambda_j-j)}
{\prod_{j=1}^{n-1}(i-\lambda_i+\lambda^*_j-j)},
$$
where $\lambda^*$ is obtained from $\lambda$ by removing a box from the $i$-th row.
\end{lem}

\begin{proof}[Proof of Theorem \ref{th:shifted}]
Let
$$
\phi(z)=
D\Bigl(\frac{\varphi_\lambda(z)}{H_\lambda}\Bigr)-\frac{z\varphi_\lambda(z+1)}{H_\lambda}.
$$
It is easy to see that $\phi(z)$ is a polynomial of $z$ with degree at most $n+1=|\lambda|+1$. Furthermore,
$$
[z^{n+1}]\ \phi(z) = \sum_{\lambda^+}\frac{1}{H_{\lambda^+}}-
\frac{1}{H_\lambda}=0
$$
and
$$
[z^{n}]\ \phi(z) =
\sum_{\lambda^+}\frac{\binom{n+2}{2}}{H_{\lambda^+}}-
\frac{1}{H_\lambda}-\frac{\binom{n+1}{2}+n}{H_\lambda}=0.
$$
This means that $\phi(z)$ is a polynomial of $z$ with degree at most $n-1$.  To show that $\phi(z)=0$, we just need to find $n$ distinct roots for $\phi(z)$. Let $z_i=i-\lambda_i-n-1$ for $1\leq i \leq n$. We will show that $\phi(z_i)=0$.

If $\lambda_i = \lambda_{i-1}$, we know the factor $z+n+1+\lambda_i-i$ lies in $\varphi_{\lambda^+}(z)$ since we can not add a box in $i$-th row to $\lambda$ and thus $\varphi_{\lambda^+}(z_i)=0$. For similar reasons, for all $1\leq i\leq n$ we have $\varphi_\lambda(z_i)=\varphi_\lambda(z_i+1)=0$, which means that $\phi(z_i)=0$.

If $\lambda_i+1 \leq \lambda_{i-1}$, we can add a box in $i$-th row to $\lambda$. First we also have $\varphi_\lambda(z_i+1)=0$ since $z_i+n+1+\lambda_i-1=0.$ To show  $\phi(z_i)=0,$ we just need to show $D(\frac{\varphi_\lambda(z_i)}{H_\lambda})=0,$ or equivalently,
$$
\sum_{\lambda^+}\frac{H_\lambda}{H_{\lambda^+}}
\varphi_{\lambda^+}(z_i)=\varphi_{\lambda}(z_i).
$$
It is easy to see that only one  term on the left side of last identity is not $0$. Thus we just need to show that
$$
\frac{H_\lambda}{H_{\lambda^{**}}}
\varphi_{\lambda^{**}}(z_i)=\varphi_{\lambda}(z_i),
$$
where $\lambda^{**}$ is obtained by adding a box to $\lambda$ in $i$-th row. But the last identity is equivalent to Lemma \ref{th:shifted3}. The proof is complete.

\end{proof}
\section{$D$-polynomials from the work of Carde-Loubert-Potechin-Sanborn} \label{sec:carde}
In this section, we derive some $D$-polynomials arising from the work of  Carde,  Loubert,  Potechin and  Sanborn \cite{clps} on one of the first author's conjecture \cite{han} related to hook lengths of partitions.  Furthermore, the degrees of such $D$-polynomials can be explicitly determined. As an application of Theorem \ref{th:main2} and Lemma \ref{th:DD-}, we obtain the skew marked hook length formula (see Theorem \ref{th:skewmarked}).

Let $z$ be a formal parameter and $\rho(h,z)$ be the function defined on each positive integer~$h$ (see \cite{clps, han}):
\begin{align*}
    \rho(h,z) &:=
\frac{(1+\sqrt z)^{h}+(1-\sqrt z)^{h}}{(1+\sqrt z)^{h}-(1-\sqrt z)^{h}}
 \cdot h{\sqrt z}\\
 &=\frac{ h \sum_{k
\geq 0} \binom{h}{2k} z^k} { \sum_{k \geq 0} \binom{h}{2k+1} z^k}\\
&=1 + \frac {h^2-1}3 z-  \frac {(h^2-1)(h^2-4)}{45} z^2 +  \frac
{(h^2-1)(h^2-4)(2h^2-11)}{945} z^3 + \cdots.
\end{align*}
\begin{defi}
The functions  $L_k(\lambda)$ of partitions are defined by the following generating function
$$
\prod_{\square\in\lambda} \rho(h_\square,z)=\sum_{k\geq
0}L_k(\lambda)z^k.
$$
\end{defi}
For example, we have
$$
L_0(\lambda)=1 \text{\quad and\quad}
L_1(\lambda)=\frac{1}{3}\sum_{\square\in\lambda}(h_\square^2-1)
=\frac {S(\lambda,1)}3.
$$

For $i=2r-1,2r,2r+1$, $D^{i}\Bigl(\frac{L_{r}(\lambda)}{H_\lambda}\Bigr)$ has an explicit expression.

\begin{thm}\label{th:DL}
For each partition $\lambda$ we have
\begin{align}
    D^{2r+1}\Bigl(\frac{L_{r}(\lambda)}{H_\lambda}\Bigr)&=0,  &(r\geq 0)\label{eq:DL+1}\\
    D^{2r}\Bigl(\frac{L_{r}(\lambda)}{H_\lambda}\Bigr)&=
    \frac{(2r-1)!!}{H_\lambda}, & (r\geq 1) \label{eq:DL+0}\\
    D^{2r-1}\Bigl(\frac{L_{r}(\lambda)}{H_\lambda}\Bigr)&=
    \frac{(2r-1)!!}{H_\lambda} |\lambda|. &(r\geq 1) \label{eq:DL-1}
\end{align}
\end{thm}

Recall the following result obtained in \cite{clps}, which will be used in the proof of Theorem \ref{th:DL}.

\begin{lem} [Carde-Loubert-Potechin-Sanborn \cite{clps}]\label{th:carde}
For each partition $ \lambda $ we have
$$
\sum_{\lambda^{+}} w(\lambda^{+}) = w(1) w(\lambda) +
\sum_{\lambda^-} w(\lambda^-),
$$
where
$$
w(\lambda) =
\prod_{\square\in\lambda} \frac{\rho(h_\square, z)}{h_\square \sqrt z}.
$$
\end{lem}

Lemma \ref{th:carde} implies
\begin{align*}
 \sum_{\lambda^{+}} \frac{\prod_{\square\in\lambda^{+}}
\rho(h_\square,z)}{H_{\lambda^{+}}} -
\frac{\prod_{\square\in\lambda} \rho(h_\square,z)}{H_\lambda}
= z
\sum_{\lambda^-} \frac{\prod_{\square\in\lambda^-}
\rho(h_\square,z)}{H_{\lambda^-}}.
\end{align*}
Comparing the coefficients of $z^k$, we obtain
\begin{equation}\label{eq:d*}
D\Bigl(\frac{L_k(\lambda)}{H_\lambda}\Bigr)=\frac{| \lambda
|
L_{k-1}(\lambda)}{H_\lambda}-D^-\Bigl(\frac{L_{k-1}(\lambda)}{H_\lambda}\Bigr).
\end{equation}

\begin{lem} \label{th:DrL}
For each partition $\lambda$ and each integer $r\geq 1$ we have
$$
D^r\Bigl(\frac{L_k(\lambda)}{H_\lambda}\Bigr)=
| \lambda | D^{r-1}\Bigl(\frac{L_{k-1}(\lambda)}{H_\lambda}\Bigr)
+(r-1)D^{r-2}\Bigl(\frac{L_{k-1}(\lambda)}{H_\lambda}\Bigr)-
D^-D^{r-1}\Bigl(\frac{L_{k-1}(\lambda)}{H_\lambda}\Bigr).
$$
\end{lem}
\begin{proof}
The lemma is true when $r=1$ by \eqref{eq:d*}. Assume that it is true for some $r\geq 1$. By Lemmas \ref{th:Dkbino} and \ref{th:DD-}
we have
\begin{align*}
    D^{r+1}\Bigl(\frac{L_{k}(\lambda)}{H_\lambda}\Bigr)&= D\Bigl(| \lambda |
    D^{r-1}\Bigl(\frac{L_{k-1}(\lambda)}{H_\lambda}\Bigr)
+(r-1)D^{r-2}\Bigl(\frac{L_{k-1}(\lambda)}{H_\lambda}\Bigr)\\
&\qquad\qquad
- D^-D^{r-1}\Bigl(\frac{L_{k-1}(\lambda)}{H_\lambda}\Bigr)\Bigr)\\
&= | \lambda |
D^{r}\Bigl(\frac{L_{k-1}(\lambda)}{H_\lambda}\Bigr)
+rD^{r-1}\Bigl(\frac{L_{k-1}(\lambda)}{H_\lambda}\Bigr)-
D^-D^{r}\Bigl(\frac{L_{k-1}(\lambda)}{H_\lambda}\Bigr).\qedhere
\end{align*}
\end{proof}

\begin{proof}[Proof of Theorem \ref{th:DL}]
Identity \eqref{eq:DL+1} is proved by induction on $r$. When $r=0$, we have $D(\frac{L_{0}(\lambda)}{H_\lambda})= D(\frac{1}{H_\lambda})=0$ by Lemma \ref{th:D+}. Assume that \eqref{eq:DL+1} is true for some $r\geq 0$. So that
$$
D^{2r+1}\Bigl(\frac{L_{r}(\lambda)}{H_\lambda}\Bigr)=
D^{2r+2}\Bigl(\frac{L_{r}(\lambda)}{H_\lambda}\Bigr)=0.
$$
\goodbreak
By Lemma \ref{th:DrL} we obtain
\begin{align*}
    \qquad D^{2r+3}\Bigl(\frac{L_{r+1}(\lambda)}{H_\lambda}\Bigr)
    &= | \lambda |
    D^{2r+2}\Bigl(\frac{L_{r}(\lambda)}{H_\lambda}\Bigr)
    +(2r+2)D^{2r+1}\Bigl(\frac{L_{r}(\lambda)}{H_\lambda}\Bigr)\\
    &\qquad \qquad - D^-D^{2r+2}\Bigl(\frac{L_{r}(\lambda)}{H_\lambda}\Bigr) \\
    &=0.
\end{align*}

For \eqref{eq:DL+0} and \eqref{eq:DL-1} we proceed in the same manner. By Lemma \ref{th:DrL}, we have
\begin{align*}
    \qquad D^{2r+2}\Bigl(\frac{L_{r+1}(\lambda)}{H_\lambda}\Bigr)
&= | \lambda |
    D^{2r+1}\Bigl(\frac{L_{r}(\lambda)}{H_\lambda}\Bigr)
    +(2r+1)D^{2r}\Bigl(\frac{L_{r}(\lambda)}{H_\lambda}\Bigr)\\
&\qquad\qquad
-
D^-D^{2r+1}\Bigl(\frac{L_{r}(\lambda)}{H_\lambda}\Bigr)\\
&=
(2r+1)D^{2r}\Bigl(\frac{L_{r}(\lambda)}{H_\lambda}\Bigr)
\\
&= (2r+1) \cdot
\frac{(2r-1)!!}{H_\lambda}\\
&=\frac{(2r+1)!!}{H_\lambda},\\
\noalign{\noindent and}
D^{2r+1}\Bigl(\frac{L_{r+1}(\lambda)}{H_\lambda}\Bigr)
& = | \lambda |
D^{2r}\Bigl(\frac{L_{r}(\lambda)}{H_\lambda}\Bigr)
+2rD^{2r-1}\Bigl(\frac{L_{r}(\lambda)}{H_\lambda}\Bigr)\\
&\qquad\qquad - D^-D^{2r}\Bigl(\frac{L_{r}(\lambda)}{H_\lambda}\Bigr)\\
&= | \lambda | \frac{(2r-1)!!}{H_\lambda} +(2r-1)!!\frac{2r|\lambda|}{H_\lambda}-
D^-\left(\frac{(2r-1)!!}{H_\lambda}\right)\\
&=
(2r+1)!!\frac{|\lambda|}{H_\lambda}.
\end{align*}
The case $r=1$ is guaranteed by Lemma \ref{th:DrL}.
\end{proof}

By Theorems \ref{th:DL} and \ref{th:main2} we obtain the following result.

\begin{thm} \label{th:polynomial*}
Let $\mu$ be a given partition and $r$ a nonnegative integer. Then
$$
\sum_{|\lambda/\mu|=n}f_{\lambda/\mu}
\frac{L_{r}(\lambda)}{H_\lambda}
=\sum_{0\leq k\leq 2r}\binom{n}{k} D^k\Bigl(\frac{L_{r}(\mu)}{H(\mu)}\Bigr)
$$
is a polynomial of $n$ with degree at most $2r$. In particular, let
$\mu=\emptyset$, we have
$$\sum_{|\lambda|=n}f_{\lambda}\frac{L_{r}(\lambda)}{H_\lambda}=
\sum_{0\leq k\leq 2r}d_{k}\binom{n}{k}$$
where
$d_k= D^k\bigl(\frac{L_{r}(\lambda)}{H_\lambda}\bigr)\big|_{\lambda=\emptyset}.$
\end{thm}

\begin{proof}[Proof of Theorem \ref{th:skewmarked}]
Let $r=1$ in Theorem \ref{th:polynomial*}. Then we obtain
\begin{align*}
\sum_{|\lambda/\mu|=n}f_{\lambda/\mu}
\frac{L_{1}(\lambda)}{H_\lambda}
&=
\frac{L_{1}(\mu)}{H(\mu)}
+
n D\Bigl(\frac{L_{1}(\mu)}{H(\mu)}\Bigr)
+
\binom{n}{2} D^2\Bigl(\frac{L_{1}(\mu)}{H(\mu)}\Bigr)\\
&=
\frac{L_{1}(\mu)}{H(\mu)}
+
n \frac{|\mu|}{H(\mu)}
+
\binom{n}{2} \frac{1}{H(\mu)},
\end{align*}
and
$$
\sum_{|\lambda/\mu|=n}f_{\lambda/\mu}
H_\mu \frac{L_{1}(\lambda) - L_1(\mu)}{H_\lambda}
=
n {|\mu|}
+
\binom{n}{2}
$$
by \eqref{eq:skewhook}. This is equivalent to \eqref{eq:skewmarked}.
\end{proof}

\section{A family of $D$-polynomials $q_\nu(\lambda)$} \label{sec:Dpoly}  
In this section, we study the properties of a family of functions $q_\nu(\lambda)$ needed in the proof of our main Theorems \ref{th:main1} and \ref{th:skewstanley}. The main result in this section is Theorem \ref{th:dqnu}.

For a partition $\lambda$, the \emph{outer corners} (see \cite{bandlow}) are the boxes which can be removed to get a new partition $\lambda^-$. Let $m=m(\la)$ be the number of outer corners of $\la$ and $(\alpha_1,\beta_1),\ldots,(\alpha_{m},\beta_{m})$ be the coordinates of outer corners such that $\alpha_1>\alpha_2>\cdots >\alpha_m$. Let $y_j=y_j(\la):=\beta_j-\alpha_j$ be the contents of outer corners for $1\leq j \leq m.$ We set $\alpha_{m+1}=\beta_0=0$ and call $(\alpha_1,\beta_0),(\alpha_2,\beta_1),\ldots,(\alpha_{m+1},\beta_{m})$ the \emph{inner corners} of $\lambda$. Let $x_i=x_i(\la):=\beta_i-\alpha_{i+1}$ be the contents of inner corners for $0\leq i \leq m$ (see Figure \ref{fig:2}). 
It is easy to verify that $x_i$ and $y_j$ satisfy the following relation:
\begin{equation}
x_0<y_1<x_1<y_2<x_2<\cdots <y_m <x_m.
\end{equation}
According to Olshanski~\cite{ols3} we define
\begin{equation}\label{def:qk}
q_k(\lambda):=\sum_{0\leq i\leq m}{x_i}^{k}-\sum_{1\leq j\leq m}{y_j}^{k}
\end{equation}
for each $k\geq 0$. The first three values of $\{q_k(\lambda)\}_{k\geq 0}$ can be evaluated explicitly. For each partition $\lambda$ we have
\begin{equation}\label{eq:q012}
q_0({\lambda})=1,\quad
q_1({\lambda})=0 \text{\quad and \quad}
q_2({\lambda})=2\,|\lambda|.
\end{equation}

\begin{figure}
\centering
\begin{center}
\includegraphics[]{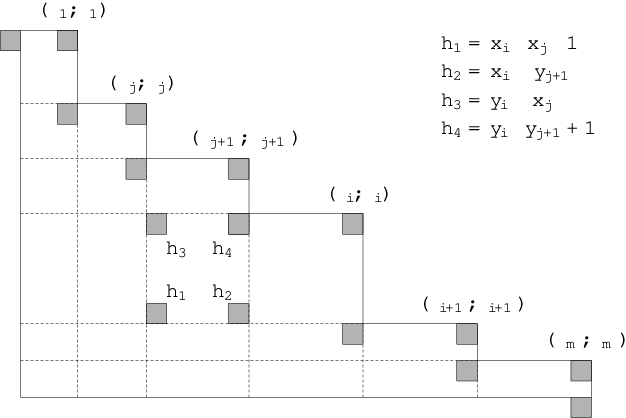}
\end{center}
\caption{A partition and its corners. The outer corners are labelled with $(\alpha_i, \beta_i)$ ($i=1,2,\ldots, m$). The inner corners are indicated by the dot symbol ``$\cdot$''.} \label{fig:2}
\end{figure}

Let us prove \eqref{eq:q012}. First we have $q_0({\lambda})=(m+1)-m=1.$ By definition of $x_i$ and $y_j$, we obtain
$$
\sum_{0\leq i\leq m}x_i =\sum_{1\leq j\leq m}y_j =\sum_{1\leq i\leq
m}\beta_{i}-\sum_{1\leq j\leq m}\alpha_{j}.
$$
Thus
$$
q_1({\lambda})=\sum_{0\leq i\leq m}x_i-\sum_{1\leq j\leq
m}y_j=0.
$$
We also have
\begin{align*}
q_2({\lambda})&=\sum_{0\leq i\leq
m}x_i^2-\sum_{1\leq j\leq m}y_j^2\\
&=\sum_{0\leq i\leq
m}(\beta_i-\alpha_{i+1})^2-\sum_{1\leq j\leq
m}(\beta_j-\alpha_{j})^2\\
&=\sum_{1\leq i\leq
m}2\beta_i(\alpha_i-\alpha_{i+1}).
\end{align*}
By Figure $2$ it is easy to see that $\sum_{1\leq i\leq m}\beta_i(\alpha_i-\alpha_{i+1})$ is equal to the number of boxes in $\lambda$, which is $|\lambda|$. Hence $q_2({\lambda})=2\,|\lambda|$.

\medskip

For each partition $\nu=(\nu_1, \nu_2, \ldots, \nu_\ell)$, the function  $q_\nu(\lambda)$ is defined by
\begin{equation}
q_\nu(\lambda):=q_{\nu_1}(\lambda)q_{\nu_2}(\lambda)\cdots q_{\nu_\ell}(\lambda).
\end{equation}

\begin{thm} \label{th:dqnu}
Let $\nu$ be a partition. Then $q_\nu(\lambda)$ is a $D$-polynomial with degree at most $|\nu|/2$. Furthermore, there exist some $b_{\delta}\in \mathbb{Q}$ such that
\begin{equation}\label{eq:Dnu}
D(\frac{q_{\nu}(\lambda)}{H_\lambda})=
\sum_{|\delta|\leq
|\nu|-2}b_{\delta}\frac{q_\delta(\lambda)}{H_\lambda}
\end{equation}
for every partition $\lambda$.
\end{thm}

Notice that \eqref{eq:Dnu} could also be obtained by carefully reading \cite{ols1} or \cite{ols3}. But for completeness and since it is not explicitly given in \cite{ols1} or \cite{ols3}, we will include a proof later.

First we prove some useful lemmas related to hook lengths and $q_{\nu}(\lambda)$. For each $k=0,1, \ldots, m$, denote by $\square_k=(\alpha_{k+1}+1,\beta_k+1)$ and $\lambda^{k+}=\lambda\cup\{\square_k\}$.
\begin{lem} \label{th:Hlambda+}
Let $g$ be a function defined on integers. Then we have
\begin{align*}
\sum_{\square\in\lambda^{k+}}
g(h_{\square})-\sum_{\square\in\lambda}g(h_{\square})
&= g(1)+\sum_{0\leq i \leq
k-1}\bigl(g(x_k-x_i)-g(x_k-y_{i+1})\bigr)\\
&\qquad \qquad +\sum_{k+1\leq i \leq m}\bigl(g(x_i-x_k)-g(y_i-x_k)\bigr)\\
\noalign{\noindent and}
\frac{\prod_{\square\in\lambda^{k+}}g(h_{\square})}{\prod_{\square\in\lambda}g(h_{\square})}
    &= g(1)\prod_{0\leq i\leq
k-1}\frac{g(x_k-x_i)}{g(x_k-y_{i+1})}\prod_{k+1\leq i\leq
m}\frac{g(x_i-x_k)}{g(y_i-x_k)}.\\
\noalign{\noindent In particular, we have}
\frac{H_{\lambda^{k+}}}{H_{{\lambda}}}&=\frac{\prod\limits_{\substack{0\leq i\leq
m\\ i\neq k}}(x_k-x_i)}{\prod\limits_{\substack{1\leq j\leq
m}}(x_k-y_j)}.
\end{align*}
\end{lem}
\begin{proof}
When adding the box $\square_k$ to $\lambda,$ it is easy to see that the hook lengths of boxes which are in the same row or the same column with $\square_k$ increase by one. The hook lengths of other boxes don't change. Thus we have
\begin{align*}
    &\sum_{\square\in\lambda^{k+}}g(h_{\square})-\sum_{\square\in\lambda}g(h_{\square})
    =
    \sum_{1\leq i\leq \alpha_{k+1}}
    \Bigl(g\bigl(h_{(i,\beta_{k}+1)}(\lambda^{k+})\bigr)-g\bigl(h_{(i,\beta_{k}+1)}(\lambda)\bigr)\Bigr)\\
&\qquad \qquad +\sum_{1\leq j\leq \beta_k}
    \Bigl(g\bigl(h_{(\alpha_{k+1}+1,j)}(\lambda^{k+})\bigr)-g\bigl(h_{(\alpha_{k+1}+1,j)}(\lambda)\bigr)\Bigr)
    +
g\bigl(h_{\square_k}(\lambda^{k+})\bigr)
,
\end{align*}
where $h_\square(\lambda)$ (resp. $h_\square(\lambda^{k+})$) denotes the hook length of the box $\square$ in $\lambda$ (resp. $\lambda^{k+}$). On the other hand, the hook lengths of
$$
(\alpha_{k+1}+1,1),(\alpha_{k+1}+1,2),\cdots,(\alpha_{k+1}+1,\beta_{k})
$$
in $\lambda$ and $\lambda^{k+}$ are
\begin{align*}
    x_k-x_i-1,x_k-x_i-2,\cdots,x_k-y_{i+1}+1,x_k-y_{i+1} \quad (0\leq i \leq k-1)
\end{align*}
and
\begin{align*}
    x_k-x_i,x_k-x_i-1,\cdots,x_k-y_{i+1}+2,x_k-y_{i+1}+1 \quad (0\leq i \leq k-1)
\end{align*}
respectively. Hence we obtain
$$\sum_{1\leq
j\leq \beta_k}
\bigl(g(h_{(\alpha_{k+1}+1,j)}(\lambda^{k+}))-g(h_{(\alpha_{k+1}+1,j)}(\lambda))\bigr)=\sum_{0\leq
i \leq k-1}\bigl(g(x_k-x_i)-g(x_k-y_{i+1})\bigr).$$

Similarly,
$$\sum_{1\leq j\leq \alpha_{k+1}}
\bigl(g(h_{(j,\beta_{k}+1)}(\lambda^{k+}))-g(h_{(j,\beta_{k}+1)}(\lambda))\bigr)=\sum_{k+1\leq
i \leq m}\bigl(g(x_i-x_k)-g(y_i-x_k)\bigr).$$

Thus we obtain the first identity in the lemma. The second follows from replacing $g(h)$ by $\ln(g(h))$. In particular, $g(h)=h$ implies the third identity.
\end{proof}

\begin{lem} \label{th:qk}
Let $g$ be a function defined on integers. Define
$$
g_1(\lambda):=\sum_{0\leq i\leq m}g(x_i)-\sum_{1\leq j\leq
m}g(y_j)
$$ 
which is a function of partitions. Then
$$
D\Bigl(\frac{g_1(\lambda)}{H_\lambda}\Bigr)=
\sum_{0\leq i\leq
m}\frac{g(x_i+1)+g(x_i-1)-2g(x_i)}{H_{\lambda^{i+}}}.
$$ 
In particular, let $g(z)=z^k$ so that $g_1(\lambda)=q_k(\lambda)$. Then we obtain
$$
D\Bigl(\frac{q_k(\lambda)}{H_\lambda}\Bigr)=
\sum_{0\leq i\leq m}
\frac 2{H_{\lambda^{i+}}}
 {\sum_{1\leq j \leq {k/2}}\binom{k}{2j}{x_{i}}^{k-2j}}.
$$
\end{lem}
\begin{proof}
Let $X=\{x_0,x_1,\ldots, x_m\}$ and $Y=\{y_1, y_2, \ldots, y_m\}$ be the sets of contents of inner corners and outer corners of $\lambda$ respectively. Four cases are to be considered. (i) If $\beta_{i}+1<\beta_{i+1}$ and $\alpha_{i+1}+1<\alpha_{i}$. Then it is easy to see that the contents of inner corners and outer corners of $\lambda^{i+}$ are $X\cup \{x_i-1, x_i+1\} \setminus \{x_i\}$ and $Y\cup \{x_i\}$ respectively. (ii) If $\beta_{i}+1=\beta_{i+1}$ and $\alpha_{i+1}+1<\alpha_{i}$, so that $y_{i+1}=x_i+1$. Hence the contents of inner corners and outer corners of $\lambda^{i+}$ are $X\cup \{x_i-1\} \setminus \{x_i\}$ and $Y\cup \{x_i\}\setminus \{x_i+1\}$ respectively. (iii) If $\beta_{i}+1<\beta_{i+1}$ and $\alpha_{i+1}+1=\alpha_{i}$, so that $y_{i}=x_i-1$. Then the contents of inner corners and outer corners of $\lambda^{i+}$ are $X\cup \{x_i+1\} \setminus \{x_i\}$ and $Y\cup \{x_i\}\setminus \{x_i-1\}$ respectively. (iv) If $\beta_{i}+1=\beta_{i+1}$ and $\alpha_{i+1}+1=\alpha_{i}$. Then $y_{i}+1=x_i=y_{i+1}-1$. The contents of inner corners and outer corners of $\lambda^{i+}$ are $X \setminus \{x_i\}$ and $Y\cup \{x_i\}\setminus \{x_i-1, x_i+1\}$ respectively. Thus we always have
\begin{equation}\label{eq:g-g}
    g_1(\lambda^{i+})-g_1(\lambda)=g(x_i+1)+g(x_i-1)-2g(x_i).
\end{equation}
Therefore
\begin{align*}
    D\Bigl(\frac{g_1(\lambda)}{H_\lambda}\Bigr)
    &=\sum_{0\leq
i\leq m}
\frac{g_1(\lambda^{i+})-g_1(\lambda)}{H_{\lambda^{i+}}}
= \sum_{0\leq i\leq m}
\frac{g(x_i+1)+g(x_i-1)-2g(x_i)}{H_{\lambda^{i+}}}
\end{align*}
by Lemma \ref{th:leibniz1}.
\end{proof}

\begin{lem}\label{th:Hlambda+1}
Let $k$ be a nonnegative integer. Then there exist some $b_{\nu}\in \mathbb{Q}$ such that
$$
\sum_{0\leq i\leq m}\frac{H_\lambda}{H_{\lambda^{i+}}}{x_i}^k=
\sum_{|\nu|\leq k}b_{\nu}q_\nu(\lambda)
$$  
for every partition $\lambda.$
\end{lem}
\begin{proof}
Let
$$
g(z)=\prod_{1\leq j\leq m}(1-y_jz) -
\sum_{0\leq i\leq m}\frac{H_\lambda}{H_{\lambda^{i+}}}
\prod_{\substack{0\leq j\leq m\\ j\neq i}}(1-x_j z).
$$
Then by Lemma \ref{th:Hlambda+} we obtain
\begin{align*}
    g\bigl(\frac{1}{x_t}\bigr)&= \prod_{1\leq j\leq m}\bigl(1-\frac{y_j}{x_t}\bigr)-
\frac{H_\lambda}{H_{\lambda^{t+}}} \prod\limits_{\substack{0\leq
j\leq m\\ j\neq t}}\bigl(1-\frac{x_j}{x_t}\bigr)
\\
&=\prod_{1\leq j\leq m}\bigl(1-\frac{y_j}{x_t}\bigr)-\frac{\prod\limits_{\substack{1\leq j\leq
m}}(x_t-y_j)}{\prod\limits_{\substack{0\leq j\leq m\\ j\neq
t}}(x_t-x_j)} \cdot \prod\limits_{\substack{0\leq j\leq m\\ j\neq
t}}\bigl(1-\frac{x_j}{x_t}\bigr)\\
&= 0.
\end{align*}
This means that $g(z)$ has at least $m+1$ roots, so that $g(z)=0$ since $g(z)$ is a polynomial of $z$ with degree at most $m$. Therefore we obtain
$$
\sum_{0\leq i\leq m}\frac{H_\lambda}{H_{\lambda^{i+}}}\cdot \frac{1}{1-x_iz}
=\frac{\prod_{1\leq j\leq m}(1-y_jz)}{\prod_{0\leq j\leq
m}(1-x_j z)},
$$ 
which means that
\begin{align*}
\sum_{0\leq i\leq
m}\frac{H_\lambda}{H_{\lambda^{i+}}}\bigl(\sum_{k\geq0}(x_iz)^k\bigr)
&= \exp\bigl(\sum_{1\leq j\leq m}\ln(1-y_jz)- \sum_{0\leq i\leq
m}\ln(1-x_iz)\bigr)\\
&= \exp\bigl(\sum_{k\geq 1}\frac{q_k(\lambda)}{k}z^k\bigr).
\end{align*}
Comparing the coefficients of $z^k$ on both sides, we obtain
$$
\sum_{0\leq i\leq m}\frac{H_\lambda}{H_{\lambda^{i+}}}{x_i}^k=
\sum_{|\nu|\leq k}b_{\nu}q_\nu(\lambda)
$$ 
for some $b_{\nu}\in \mathbb{Q}$. Notice that $b_{\nu}$ are independent of $\lambda$. This achieves the proof.
\end{proof}

Now we will give a proof of Theorem \ref{th:dqnu}. 

\begin{proof}[Proof of Theorem \ref{th:dqnu}]
Let $k$ be an integer. By Lemma \ref{th:qk} we have
\begin{align*}
    H_\lambda D\bigl(\frac{q_k(\lambda)}{H_\lambda}\bigr)=
 \sum_{0\leq i\leq m}  \frac{H_\lambda}{H_{\lambda^{i+}}}
\sum_{1\leq j \leq{k/2}}2\binom{k}{2j}{x_{i}}^{k-2j}.
\end{align*}
Then there exist some
$b_{\delta}\in \mathbb{Q}$ such that
$$
D\bigl(\frac{q_k(\lambda)}{H_\lambda}\bigr)=
\sum_{|\delta|\leq
k-2}b_{\delta}\frac{q_\delta(\lambda)}{H_\lambda}
$$  
for every partition $\lambda$ by Lemma \ref{th:Hlambda+1}. In other words, \eqref{eq:Dnu} is true for $\nu=(k)$.

From \eqref{eq:g-g} with $g(z)=z^k$ we actually obtain
\begin{align*}
q_k(\lambda^{i+})-q_k(\lambda)= \sum_{1\leq j
\leq{k/2}}2\binom{k}{2j}{x_{i}}^{k-2j},
\end{align*}
which is a polynomial of $x_i$ with degree at most $k-2$. Then by Lemmas \ref{th:leibniz} and~\ref{th:Hlambda+1} there exist some $b_{\delta}\in \mathbb{Q}$ such that
$$
H_\lambda D(\frac{q_{\nu}(\lambda)}{H_\lambda})=
\sum_{|\delta|\leq
|\nu|-2}b_{\delta}q_\delta(\lambda)
$$ 
for every partition $\lambda.$
\end{proof}

\section{Hook lengths and $D$-polynomials} \label{sec:main1}  
In this section, we prove the main Theorems \ref{th:main1} and \ref{th:skewstanley}. 

Let $r$ be a fixed nonnegative integer. The key step is to show that $S(\lambda, r)$ defined in \eqref{def:S} can be written as a symmetric polynomial on $\{x_0,x_1,\ldots, x_m\}$ and $\{y_1,y_2,\ldots, y_m\}$, as stated next.

\begin{thm} \label{th:Ssym}
    There exist some rational numbers $b_{\nu}=b_\nu(r)$
indexed by integer partitions $\nu$
such that
\begin{equation}
S(\lambda,r)=\sum_{|\nu|\leq 2r+2}b_{\nu}q_\nu(\lambda)
\end{equation}
for every partition $\lambda$.
\end{thm}

\medskip

Keep the same notations as in Section~\ref{sec:Dpoly} (see Figure \ref{fig:2}). Let
$$
A_{ij}=
\{ (i',j')\in \lambda:
\alpha_{i+1}+1\leq i'\leq \alpha_{i},\beta_{j}+1\leq j'\leq
\beta_{j+1} \}
$$
so that
$$
\lambda=\bigcup_{0\leq j<i\leq m}A_{ij}.
$$
The multiset of hook lengths of $A_{ij}$ are
$$
\bigcup_{a=x_{i}-y_{j+1}}^{x_{i}-x_{j}-1} \{a,a-1,a-2,\ldots,
a-(x_{i}-y_{i}-1)  \}.
$$

Let $F_0(n)$ be a function defined on integers. Define
$$
F_1(n):=\sum_{k=1}^nF_0(k) \text{\qquad and\qquad}
F_2(n):=\sum_{k=1}^nF_1(k).
$$ 
Hence
\begin{align*}
\sum_{\square \in A_{ij}}F_0(h_{\square}) &=\sum_{a=x_{i}-y_{j+1}}^{
x_{i}-x_{j}-1}\sum_{b=0}^{x_{i}-y_{i}-1}F_0(a-b)
\\
&=\sum_{a=x_{i}-y_{j+1}}^{
x_{i}-x_{j}-1}\bigl(F_1(a)-F_1(a-x_{i}+y_{i})\bigr)
\\
&=\sum_{a=x_{i}-y_{j+1}}^{
x_{i}-x_{j}-1}F_1(a)-\sum_{a=x_{i}-y_{j+1}}^{
x_{i}-x_{j}-1}F_1(a-x_{i}+y_{i})
\\
&=F_2(x_{i}-x_{j}-1)+F_2(y_{i}-y_{j+1}-1)\\
&\qquad\qquad-F_2(x_{i}-y_{j+1}-1)-F_2(y_{i}-x_{j}-1)
\end{align*}

and thus

\begin{align}\sum_{\square \in
\lambda}F_0(h_{\square})
&=\sum_{0\leq j<i\leq m}\sum_{\square \in A_{ij}}F_0(h_{\square})\nonumber\\
&=\sum_{0\leq j<i\leq
m}\bigl(F_2(x_{i}-x_{j}-1)+F_2(y_{i}-y_{j+1}-1)\label{eq:sumF}\\
&\qquad\qquad -F_2(x_{i}-y_{j+1}-1)-F_2(y_{i}-x_{j}-1)\bigr).\nonumber
\end{align}

For each $n\geq 1$ the polynomial $P_n(z)$ of real number $z$ is defined by
$$
P_n(z):={z^{n+1}\over n+1}
+{z^n\over 2}
+{1\over n+1}\sum_{1\le j\le n/2} {n+1\choose 2j}
z^{n-2j+1} (-1)^{j+1} B_{2j},
$$
where $B_{2j}$ are Bernoulli numbers \cite{conguy, foatahan, knuth}. Let $k$ be a positive integer. According to Euler-MacLaurin formula \cite{knuth},
$$
P_n(k)=1^n+2^n+\cdots + k^n.
$$
Consequently, $P_n(k)=P_n(k+1)-(k+1)^n$. It is easy to obtain the following identity:
\begin{equation}\label{eq:Pneg}
    P_n(-k-1)=(-1)^{n+1} P_n(k). \qquad (n\geq 1)
\end{equation}
For simplicity we rewrite
\begin{equation}\label{eq:Pnocoeff}
P_n(z)=
{z^n\over 2}
+\sum_{0\le j\le n/2} \zeta_j(n) z^{n-2j+1}.
\end{equation}

Let $G_0(j)=\prod_{1\leq i\leq r}(j^2-i^2)=\sum_{w=0}^r \eta_w j^{2w}.$ We define
$$
G_1(n):=\sum_{k=1}^nG_0(k) \text{\qquad and \qquad} G_2(n):=\sum_{k=1}^nG_1(k).
$$
The polynomial $G(z)$ of real number $z$ is defined by
\begin{equation}\label{def:G}
G(z):= (-1)^r  \frac {z^2r!^2}2 +
\sum_{w=1}^r \eta_w \Bigl(\frac{P_{2w}(z-1)}2 + \sum_{j=0}^w \zeta_j(2w) P_{2w-2j+1}(z-1)\Bigr).
\end{equation}

\begin{lem} \label{th:G012}
The function $G(z)$ defined in \eqref{def:G} satisfies the following relations:
\begin{align}
G(0)&=0,\\
    G(n)&=(-1)^r \frac{nr!^2}2 + G_2(n-1), \label{eq:Gn} \qquad (n\in \mathbb{N})\\
    G(n)&=G(-n). \qquad (n\in \mathbb{N})
\end{align}
\end{lem}

\begin{proof}
It's obvious that $P_n(0)=0$ and thus $P_n(-1)=0$ by \eqref{eq:Pneg}. So that $G(0)=0$ follows from \eqref{def:G}. By definitions of $G_0, G_1$ and $G_2$ we have
\begin{align*}
    G_2(n-1) &= \sum_{k=1}^{n-1} \sum_{j=1}^k \sum_{w=0}^r \eta_w j^{2w}\\
    &= \sum_{k=1}^{n-1} \sum_{j=1}^k  \eta_0 +\sum_{w=1}^r \eta_w \sum_{k=1}^{n-1} P_{2w}(k) \\
    &= \eta_0\binom n2 + \sum_{w=1}^r \eta_w \sum_{k=1}^{n-1} \Bigl(\frac {k^{2w}}2
        + \sum_{j=0}^w \zeta_j(2w) k^{2w-2j+1}\Bigr) \\
            &=  (-1)^r r!^2\binom n2+\sum_{w=1}^r \eta_w
            \Bigl(\frac {P_{2w}(n-1)}2 + \sum_{j=0}^w \zeta_j(2w) P_{2w-2j+1}(n-1)\Bigr). \\
\end{align*}
Hence \eqref{eq:Gn} is true. By \eqref{eq:Pneg},
\begin{align*}
    G(n)-G(-n)
    &=
\sum_{w=1}^r \eta_w \Bigl(\frac{P_{2w}(n-1)}2 + \sum_{j=0}^w \zeta_j(2w) P_{2w-2j+1}(n-1)\Bigr)\\
&\qquad -
\sum_{w=1}^r \eta_w \Bigl(-\frac{P_{2w}(n)}2 + \sum_{j=0}^w \zeta_j(2w) P_{2w-2j+1}(n)\Bigr)\\
&=
\sum_{w=1}^r \eta_w \Bigl(P_{2w}(n)-\frac{n^{2w}}2 - \sum_{j=0}^w \zeta_j(2w) n^{2w-2j+1}\Bigr)\\
&=0.\qedhere
\end{align*}

\end{proof}

The above lemma implies that $G(n)$ is an even polynomial of the integer $n$ with degree $2r+2$, which means that there exist some rational numbers $\xi_i$ such that
\begin{equation}\label{eq:Geven}
G(n)=\sum_{i=1}^{r+1} \xi_i n^{2i}.
\end{equation}

\begin{proof}[Proof of Theorem \ref{th:Ssym}]
By \eqref{eq:sumF} we obtain
\begin{align*}
    S(\lambda,r)&=\sum_{\square\in \lambda}G_0(h_\square)\\
&= \sum_{0\leq j<i\leq
m}\bigl(G_2(x_{i}-x_{j}-1)+G_2(y_{i}-y_{j+1}-1)\\
&\qquad\qquad -G_2(x_{i}-y_{j+1}-1)-G_2(y_{i}-x_{j}-1)\bigr)\\
&= \sum_{0\leq j<i\leq
m}\bigl(G(x_{i}-x_{j})+G(y_{i}-y_{j+1})-G(x_{i}-y_{j+1})-G(y_{i}-x_{j})\bigr).
 \end{align*}
The last equality is due to \eqref{eq:Gn} and
$$
(x_{i}-x_{j})+(y_{i}-y_{j+1})-(x_{i}-y_{j+1})-(y_{i}-x_{j})=0.
$$
Thus by \eqref{eq:Geven}, we have
\begin{align*}
S(\lambda,r)&= \sum_{1\leq k\leq r+1}\xi_k\sum_{0\leq j<i\leq
m}\bigl((x_i-x_j)^{2k}+(y_{i}-y_{j+1})^{2k}\\
&\qquad\qquad -(x_i-y_{j+1})^{2k}-(y_{i}-x_j)^{2k}\bigr)\\
&= \sum_{1\leq k\leq
r+1}\xi_k V(k),
 \end{align*}
 where
 \begin{align*}
V(k)&= \sum_{0\leq i\leq j\leq m}(x_i-x_j)^{2k}+\sum_{1\leq i\leq
j\leq m}(y_{i}-y_j)^{2k}
 -\sum_{0\leq i\leq m}\sum_{1\leq j\leq m}(x_i-y_j)^{2k}.
 \end{align*}
Notice that $\xi_k$ is independent of $\lambda$ since  $G(n)$ is independent of $\lambda$. Comparing the coefficients of $z^{2k}$ $(1\leq k \leq r+1)$ on both sides of the following trivial identity
\begin{align*}
    &\Bigl(\sum_{i=0}^m e^{x_iz}-\sum_{j=1}^m e^{y_jz}\Bigr)
    \Bigl(\sum_{i=0}^m e^{-x_iz}-\sum_{j=1}^m e^{-y_jz}\Bigr) \\
& =\sum_{i=0}^m\sum_{j=0}^m e^{(x_i-x_j)z}+
\sum_{i=1}^m\sum_{j=1}^m e^{(y_{i}-y_j)z}
 -\sum_{i=0}^m\sum_{j=1}^m e^{(x_i-y_j)z}
-\sum_{i=0}^m\sum_{j=1}^m e^{(y_j-x_i)z},
\end{align*}
we obtain there exist some rational numbers $b_\nu'$ such that
\begin{equation}
    V(k)=\sum_{|\nu|\leq 2k}b_{\nu}' q_\nu(\lambda)
\end{equation}
for every partition $\lambda$. This achieves the proof.
\end{proof}

For each partition $\nu=(\nu_1,\nu_2,\cdots,\nu_\ell)$ we define
$$
S_\nu(\lambda):=\prod_{1\leq i\leq \ell}S(\lambda,\nu_i).
$$
Combining Theorems \ref{th:Ssym}  and \ref{th:dqnu} we derive the following result.

\begin{thm} \label{th:SDpoly}
Let $\nu=(\nu_1,\nu_2,\cdots,\nu_\ell)$ be a given partition. Then $S_\nu(\lambda)$ is a $D$-polynomial with degree at most $|\nu|+\ell$. Furthermore, there exist some $b_{\delta}\in \mathbb{Q}$ indexed by partitions $\delta$ such that
\begin{equation}\label{eq:DkS}
D^k\Bigl(\frac{S_\nu(\lambda)}{H_\lambda}\Bigr)=\sum_{|\delta|\leq
2|\nu|+2\ell-2k}b_{\delta}\frac{q_\delta(\lambda)}{H_\lambda}
\end{equation}
for every partition $\lambda.$
\end{thm}

Now we are ready to prove Theorems \ref{th:main1} and \ref{th:skewstanley}.
\begin{proof}[Proof of Theorem \ref{th:main1}]
Notice that $p_\nu(h_{\square}^2: {\square}\in\lambda)$ can be written as a linear combination of some $S_\nu(\lambda)$.  Then by Theorem \ref{th:SDpoly} we obtain Theorem \ref{th:main1}.
\end{proof}

\begin{proof}[Proof of Theorem \ref{th:skewstanley}]
It is easy to see that for any symmetric function $F(z_1, z_2, \ldots)$ of infinite variables, $F(h_{\square}^2:\square\in \lambda)$ can be written as a linear combination of some $p_\nu(h_{\square}^2: {\square}\in\lambda)$.  Then by Theorems \ref{th:main1}, \ref{th:main2} and \ref{th:SDpoly} we derive Theorem \ref{th:skewstanley}.
\end{proof}

\section{Okada-Panova hook length formula} \label{sec:okadapanova} 

Okada's conjecture on hook lengths \eqref{eq:okadapanova} was first proved by Panova \cite{panova} by means of Theorem \ref{th:hanstanley}. In this section, we generalize and give another proof of the Okada-Panova hook length formula by using difference operators. In fact, the constants $K_r$ arise directly from the computation for a single partition $\lambda$, without the summation ranging over all partitions of size $n$.

\begin{proof}[Proof of Corollary \ref{th:DS}]
By \eqref{eq:q012} and Theorem \ref{th:SDpoly} there exist $a,b\in \mathbb{Q}$ such that for every $\lambda,$
$$
H_\lambda D^{r}\Bigl(\frac{S(\lambda,r)}{H_\lambda}\Bigr)=a|\lambda|+b.
$$
The explicit values of $a$ and $b$ are determined by taking two special partitions $\lambda=\emptyset$ and $\lambda=(1)$. Since $S(\lambda,r)=0$ if $\lambda$ does not have any hook length greater than $r$, we have
\begin{equation*}
b=D^{r}\Bigl(\frac{S(\lambda,r)}{H_\lambda}\Bigr)\Big|_{\lambda=\emptyset}=0
\end{equation*}
by \eqref{eq:main2b}. On the other hand, it's obvious that the only partitions of size $r+1$ who have hook lengths greater than $r$ are $\{ \lambda^{(k)}:0\leq k\leq r\}$ where
$$
\lambda^{(k)}=(k+1,\underbrace{1,1,\cdots,1}_{r-k}).
$$
Then
$$
f_{\lambda^{(k)}}=\binom rk
\text{\qquad and\qquad}
S(\lambda^{(k)},r)=\prod_{1\leq i\leq r}\bigl((r+1)^2-i^2\bigr).
$$
By \eqref{eq:main2b} we have 
\begin{equation*}
    a=D^{r}\Bigl(\frac{S(\lambda,r)}{H_\lambda}\Bigr)\Big|_{\lambda=(1)}
= \sum_{| \lambda|=r+1}
f_{\lambda}\frac{S(\lambda,r)}{H_\lambda}
= \sum_{0\leq k\leq
r}f_{\lambda^{(k)}}\frac{S(\lambda^{(k)},r)}{H_{\lambda^{(k)}}},
\end{equation*}
so that
\begin{equation*}
a= \frac{(2r+1)!}
 {r!(r+1)^2}\sum_{0\leq k\leq r}{\binom{r}{k}}^2
 = \frac{(2r+1)!}
 {r!(r+1)^2}\binom{2r}{r}
 =K_r.
\end{equation*}
Hence \eqref{eq:DS0} is true. Consequently, \eqref{eq:DS1} and \eqref{eq:DS2} are derived from \eqref{eq:DS0} by applying the difference operator $D$.
\end{proof}

\begin{proof}[Proof of Theorem \ref{th:okadapanova}]
Since $S(\lambda,r)=0$ if $\lambda$ does not have any hook length greater than $r$, we have
\begin{equation}\label{eq:S0ir}
D^{i}\Bigl(\frac{S(\lambda,r)}{H_\lambda}\Bigr)\Big|_{\lambda=\emptyset}=0
\end{equation}
for $0\leq i\leq r$  by \eqref{eq:main2b}. Substituting $g(\lambda)$ by $S(\lambda,r)/H_\lambda$ and $\mu$ by $\emptyset$ in \eqref{eq:main2} we get
\begin{equation*}
    \sum_{|\lambda|=n}f_{\lambda} \frac{ S(\lambda,r)}{H_\lambda}
=\sum_{k=0}^n\binom{n}{k} D^{k}\Bigl(\frac{S(\mu,r)}{H(\mu)}\Bigr)\Big|_{\mu=\emptyset}
=K_r \binom{n}{r+1}
\end{equation*}
by \eqref{eq:S0ir}, \eqref{eq:DS1} and \eqref{eq:DS2}.
\end{proof}

\section{Fujii-Kanno-Moriyama-Okada content formula} \label{sec:content} 
In this section, we prove and generalize the Fujii-Kanno-Moriyama-Okada content formula. 
Recall $C(\lambda,r)=\sum_{\square\in \lambda
}\prod_{i=0}^{r-1}(c_{\square}^2 - i^2)$.

\begin{thm} \label{th:Csym}
There exist some $b_{\nu}\in \mathbb{Q}$ indexed by partitions $\nu$ such that
$$
    H_\lambda D\Bigl(\frac{C(\lambda,r)}{H_\lambda}\Bigr) = \sum_{|\nu|\leq
2r}b_{\nu}q_\nu(\lambda)
$$ 
for every partition $\lambda$.
\end{thm}
\begin{proof}
We have
$$
\sum_{\square\in \lambda^{i+}}
c_{\square}^{2r}-\sum_{\square\in \lambda}
c_{\square}^{2r}=(\beta_i-\alpha_{i+1})^{2r}=x_i^{2r}.
$$
Therefore
\begin{align*} H_\lambda D\Bigl(\frac{\sum_{\square\in
\lambda} c_{\square}^{2r}}{H_\lambda}\Bigr)
=
            \sum_{\lambda^{i+}}\frac{H_\lambda}{H_{\lambda^{i+}}}
            \Bigl(\sum_{\square\in
\lambda^{i+}} c_{\square}^{2r}-\sum_{\square\in \lambda}
c_{\square}^{2r}\Bigr)
=\sum_{\lambda^{i+}}
\frac{H_\lambda}{H_{\lambda^{i+}}}x_i^{2r}.
\end{align*}
The proof is achieved by Lemma \ref{th:Hlambda+1} and linearity.
\end{proof}

\begin{proof}[Proof of Theorem \ref{th:DC}]
    By \eqref{eq:q012},  Theorems \ref{th:Csym} and  \ref{th:dqnu} there exist
$a,b\in \mathbb{Q}$ such that for every $\lambda,$
$$
H_\lambda D^{r}\Bigl(\frac{C(\lambda,r)}{H_\lambda}\Bigr)=a|\lambda|+b.
$$
The explicit values of $a$ and $b$ are determined by taking two special partitions $\lambda=\emptyset$ and $\lambda=(1)$. Since $C(\lambda,r)=0$ if $\lambda$ does not have any content whose absolute value is greater than $r-1$, we have
\begin{equation*}
b=D^{r}\Bigl(\frac{C(\lambda,r)}{H_\lambda}\Bigr)\Big|_{\lambda=\emptyset}=0
\end{equation*}
by \eqref{eq:main2b}. On the other hand, it's obvious that the only partitions of size $r+1$ who have contents with absolute values greater than $r-1$ are $(1^{r+1})$ and $(r+1)$. By \eqref{eq:main2b} we have
\begin{equation*}
    a=D^{r}\Bigl(\frac{C(\lambda,r)}{H_\lambda}\Bigr)\Big|_{\lambda=(1)}
= \sum_{| \lambda|=r+1}
f_{\lambda}\frac{C(\lambda,r)}{H_\lambda}
= \frac{(2r)!}{(r+1)!}.
\end{equation*}
Hence \eqref{eq:DC0} is true. Consequently, \eqref{eq:DC1} and \eqref{eq:DC2} are derived from \eqref{eq:DC0} by applying the difference operator $D$.
\end{proof}

\begin{proof}[Proof of Theorem \ref{th:content}]
Since $C(\lambda,r)=0$ if $\lambda$ does not have any content whose absolute value is greater than $r-1$, we have
\begin{equation}\label{eq:C0ir}
D^{i}\Bigl(\frac{C(\lambda,r)}{H_\lambda}\Bigr)\Big|_{\lambda=\emptyset}=0
\end{equation}
for $0\leq i\leq r$  by \eqref{eq:main2b}. Substituting $g(\lambda)$ by $C(\lambda,r)/H_\lambda$ and $\mu$ by $\emptyset$ in \eqref{eq:main2} we get
\begin{equation*}
    \sum_{|\lambda|=n}f_{\lambda} \frac{ C(\lambda,r)}{H_\lambda}
=\sum_{k=0}^n\binom{n}{k} D^{k}\Bigl(\frac{C(\mu,r)}{H(\mu)}\Bigr)\Big|_{\mu=\emptyset}
=\binom{(2r)!}{(r+1)!} \binom{n}{r+1}
\end{equation*}
by \eqref{eq:C0ir}, \eqref{eq:DC1} and \eqref{eq:DC2}.
\end{proof}

Theorem \ref{th:skewmarkedcontent} is a simple consequence of Theorems \ref{th:main2} and \ref{th:DC}.

\goodbreak

\bigskip
{\bf Acknowledgements}. 
The authors really appreciate the valuable suggestions given by referees for improving the overall quality of the manuscript.  The second author would like to thank Prof. P.~O.~Dehaye for the helpful guidance.

\medskip
{\bf Fundings}. 
The second author is supported  by Grant [P2ZHP2\_171879] of the Swiss National Science Foundation and the Post-doctoral Fellowship from LABEX of the University of Strasbourg.




\begin{thebibliography}{1} 


\bibitem{tew} T. Amdeberhan, Differential operators, shifted parts, and hook lengths, \emph{Ramanujan J.} \textbf{24(3)}(2011), 259--271.



\bibitem{bandlow} J. Bandlow, An elementary proof of the Hook
formula, \emph{Electron. J. Combin.} \textbf{15}(2008), research paper 45.

\bibitem{clps} K. Carde, J. Loubert, A. Potechin, and A. Sanborn, Proof of Han's Hook Expansion Conjecture, preprint; {\tt arXiv:0808.0928}.

\bibitem{conguy} J. Conway and R. Guy,  \emph{The Book of Numbers}, New York, Springer-Verlag, 1996.

\bibitem{hanxiong2} P.-O. Dehaye, G.-N. Han, and H. Xiong, Difference operators for partitions under the Littlewood decomposition, \emph{Ramanujan J.} \textbf{44(1)}(2017), 197--225.


\bibitem{foatahan} D. Foata and G.-N. Han, \emph{Principes de combinatoire classique (online), (Cours et exercices corrig\'es)}. Niveau master de math\'ematiques,
2000.

\bibitem{frt} J. S. Frame, G. de B. Robinson, and R. M. Thrall, The hook graphs of $S_n$, \emph{Canad.\ J.\ Math.}\ \textbf{6}(1954), 316--324.

\bibitem{fkmo}  S. Fujii, H. Kanno, S. Moriyama, and S. Okada, Instanton calculus and chiral one-point functions in supersymmetric gauge theories, \emph{Adv.\ Theor.\ Math.\ Phys.} \textbf{12(6)}(2008), 1401--1428.


\bibitem{han2} G.-N. Han, The Nekrasov-Okounkov hook length formula: refinement, elementary proof, extension, and applications, \emph{Ann. Inst. Fourier} \textbf{60(1)}(2010), 1--29.

\bibitem{han} G.-N. Han, Some conjectures and open problems on partition hook lengths, \emph{Experimental Mathematics} \textbf{18}(2009),  97--106.

\bibitem{han3} G.-N. Han, Hook lengths and shifted parts of partitions, \emph{Ramanujan J.} \textbf{23(1-3)}(2010),  127--135.

\bibitem{han4} G.-N. Han and K. Q. Ji, Combining hook length formulas and BG-ranks for partitions via the Littlewood decomposition, \emph{Trans. Amer. Math. Soc.} \textbf{363}(2011),  1041--1060.

\bibitem{hanxiong3} G.-N.~Han and H.~Xiong, New hook-content formulas for strict partitions, {\em Algebraic Combin.} \textbf{45(4)}(2017), 1001--1019.

\bibitem{KJ}
K. Johansson, Discrete orthogonal polynomial ensembles and the Plancherel measure, \emph{Ann. Math.} \textbf{153(1)}(2001), 259--296.

\bibitem{knuth} D. Knuth and T. Buckholtz,
Computation of tangent, Euler, and Bernoulli numbers. \emph{Math. Comp.}  \textbf{21}(1967), 663--688.



\bibitem{rsk} D. Knuth, \emph{The Art of Computer Programming, Vol. 3: Sorting and Searching}, Addison--Wesley, London, $1973$.

\bibitem{lascoux} A. Lascoux, \emph{Symmetric Functions and
    Combinatorial Operators on Polynomials}, CBMS Regional Conference Series in Mathematics, no.~99, American Mathematical Society, Providence, RI, 2003.
    
\bibitem{Macdonald} I. G. Macdonald, \emph{Symmetric functions and Hall polynomials}, Oxford Mathematical Monographs, The Clarendon Press, Oxford University Press, New York, second edition, 1995.    

\bibitem{no} N. A. Nekrasov and A. Okounkov, Seiberg-Witten theory and random partitions, in The unity of mathematics, \emph{Progress in Mathematics} \textbf{244}, Birkh\"auser Boston, 2006, pp.~525--596.


\bibitem{OEIS:A204515}
    OEIS Foundation,
    \newblock Sequence {A}204515,
    \newblock {\em The On-Line Encyclopedia of Integer Sequences}, 2015.

\bibitem{ols1} G. Olshanski,
Anisotropic Young diagrams and infinite-dimensional diffusion processes with the Jack parameter, \emph{Int. Math. Res. Not. IMRN} \textbf{6}(2010),  1102--1166.

\bibitem{ols2} G. Olshanski,
Laguerre and Meixner symmetric functions, and infinite-dimensional diffusion processes, \emph{J. Math. Sci.} \textbf{174(1)}(2011),  41--57.


\bibitem{ols3} G. Olshanski, Plancherel averages: Remarks on a paper by Stanley, \emph{Electron. J. Combin.} \textbf{17}(2010), research paper 43.

\bibitem{panova} G. Panova,
Polynomiality of some hook-length statistics. \emph{Ramanujan J.} \textbf{27(3)}(2012),  349--356.

\bibitem{petrov} L. Petrov, ${\germ{sl}}(2)$ operators and {M}arkov processes on branching graphs,  \emph{J. Algebraic Combin.} \textbf{38(3)}(2013), 663--720.

\bibitem{Rota1964}
G.-C. Rota, On the foundations of combinatorial theory. {I}. {T}heory of {M}\"obius functions, \emph{Z. Wahrscheinlichkeitstheorie und Verw. Gebiete} \textbf{2}(1964), 349--356.

\bibitem{stan} R. P. Stanley, Some combinatorial properties of hook lengths, contents, and parts of partitions. \emph{Ramanujan J.} \textbf{23(1-3)}(2010), 91--105.

\bibitem{stan1}  R. P. Stanley, Differential posets. \emph{J. Amer. Math. Soc. } \textbf{1(4)}(1988), 919--961.



\bibitem{ec2} R. P. Stanley, \emph{Enumerative Combinatorics}, vol.~2, Cambridge University Press, New York/Cam\-bridge, 1999.




\end{thebibliography}
\end{document}